\documentclass[final,twoside,a4paper]{article}

\usepackage{amsmath,amsfonts,amssymb,amsthm}
\usepackage{graphicx}
\usepackage{subfigure,color}
\usepackage{mathptmx}
\usepackage{colonequals,stmaryrd}
\usepackage{geometry}
\usepackage[colorlinks=true,citecolor=blue]{hyperref}

\hypersetup{%
pdftitle={Semi-discrete finite difference multiscale scheme for a concrete corrosion model: approximation estimates and convergence},
pdfauthor={Vladimir Chalupecky, Adrian Muntean},
pdfkeywords={Multiscale reaction-diffusion equations, Two-scale finite difference method, Approximation of weak solutions, Convergence, Concrete corrosion}
}

\DeclareMathOperator{\Div}{\mathrm{div}}
\newcommand{\coleq}{\colonequals}
\newcommand{\R}{\mathbb{R}} % real numbers
\newcommand{\PhO}{\mathcal{P}_h^1}
\newcommand{\Pho}{\mathcal{P}_h^2}
\newcommand{\Eh}{\mathcal{E}_h}
\newcommand{\Fh}{\mathcal{F}_h}
\newcommand{\Gh}{\mathcal{G}_h}
\newcommand{\Gho}{\mathcal{G}_h^o}
\newcommand{\Hh}{\mathcal{H}_h}
\newcommand{\dt}{\mathrm{d}t}
\newcommand{\vph}{\varphi_h}
\newcommand{\Omegah}{\Omega_h}
\newcommand{\Omegaho}{\Omega_h^o}
\newcommand{\Omegahe}{\Omega_h^e}
\newcommand{\omegah}{\omega_h}
\newcommand{\omegahe}{\omega_h^e}
\newcommand{\mrm}[1]{\mathrm{#1}}

\newcommand{\deriv}[2]{\frac{\mrm{d}#1}{\mrm{d}#2}}

\newcommand{\triangledl}{\triangleleft}
\newcommand{\triangleur}{\triangleright}

\newcommand{\dualsymbol}{\scriptscriptstyle\boxdot}
\newcommand{\simplsymbol}{\scriptscriptstyle\boxbslash}
\newcommand{\OmegahD}{\Omegah^{\dualsymbol}}
\newcommand{\OmegahS}{\Omegah^{\simplsymbol}}
\newcommand{\omegahD}{\omegah^{\dualsymbol}}
\newcommand{\omegahS}{\omegah^{\simplsymbol}}
\newcommand{\KD}{\mathcal{K}^{\dualsymbol}}
\newcommand{\KS}{\mathcal{K}^{\simplsymbol}}
\newcommand{\LD}{\mathcal{L}^{\dualsymbol}}
\newcommand{\LSDL}{\mathcal{L}^{\triangledl}}
\newcommand{\LSUR}{\mathcal{L}^{\triangleur}}

\newcommand{\weakto}{\rightharpoonup}

\theoremstyle{plain}
\newtheorem{thm}{Theorem}
\newtheorem{lem}[thm]{Lemma}
\newtheorem{prop}[thm]{Proposition}

\theoremstyle{definition}
\newtheorem{defn}[thm]{Definition}

\theoremstyle{remark}
\newtheorem{rem}{Remark}

\begin{document}

\title{Semi-discrete finite difference multiscale scheme for a concrete corrosion model: approximation estimates and convergence}

% \titlerunning{Semi-discrete finite difference multiscale scheme} % if too long for running head

\author{Vladim\'ir Chalupeck\'y\footnote{%
Institute of Mathematics for Industry, Kyushu University, 744 Motooka, Nishi-ku,
819-0395 Fukuoka, Japan.\newline
E-mail:~\url{chalupecky@imi.kyushu-u.ac.jp}}
\and
Adrian Muntean\footnote{%
Faculty of Mathematics and Computer Science, CASA - Center for Analysis, Scientific computing and
Applications, Institute for Complex Molecular Systems (ICMS),
Technical University of Eindhoven, PO Box 513, 5600 MB, Eindhoven, The
Netherlands.\newline
E-mail:~\url{a.muntean@tue.nl}}}

\date{June 24, 2011}

\maketitle

\begin{abstract}
We propose a semi-discrete finite difference multiscale scheme for a concrete
corrosion model consisting of a system of two-scale reaction-diffusion equations
coupled with an ode. We prove energy and regularity estimates and use them to
get the necessary compactness of the approximation estimates. Finally, we
illustrate numerically the behavior of the two-scale finite difference
approximation of the weak solution.

\vspace{1em}
\noindent
\textbf{Keywords\ } Multiscale reaction-diffusion equations\ \textperiodcentered\
Two-scale finite difference method\ \textperiodcentered\ Approximation of weak
solutions\ \textperiodcentered\ Convergence\ \textperiodcentered\ Concrete corrosion

\vspace{1em}
\noindent
\textbf{Mathematics Subject Classification (2000)\ }
MSC 
35K51\ \textperiodcentered\ %Initial-boundary value problems for second-order parabolic systems
35K57\ \textperiodcentered\ %Reaction-diffusion equations
65M06\ \textperiodcentered\ %(Finite difference method),
65M12\ \textperiodcentered\ %(Stability and convergence of numerical methods),
65M20 %(Method of lines)
\end{abstract}

\section{Introduction} \label{intro}

Biogenic sulfide corrosion of concrete is a bacterially mediated process of
forming hydrogen sulfide gas and the subsequent conversion to sulfuric acid that
attacks concrete and steel within wastewater environments. The hydrogen sulfide
gas is oxidized in the presence of moisture to form sulfuric acid that attacks
the matrix of concrete. The effect of sulfuric acid on concrete and steel
surfaces exposed to severe wastewater environments (like sewer pipes) is
devastating, and  is always associated with high maintenance costs.

The process can be briefly described as follows: Fresh domestic sewage entering
a wastewater collection system contains large amounts of  sulfates that, in the
absence of dissolved oxygen and nitrates, are reduced by  bacteria. Such
bacteria identified primarily from the  anaerobic species Desulfovibrio lead to
the fast formation of hydrogen sulfide (${\rm H_2S}$) via a complex pathway of
biochemical reactions. Once the gaseous ${\rm H_2S}$ diffuses into the headspace
environment above the wastewater, a sulfur oxidizing bacteria -- primarily
Thiobacillus aerobic bacteria -- metabolizes the ${\rm H_2S}$ gas and oxidize it
to sulphuric acid. It is worth noting that Thiobacillus  colonizes on pipe
crowns above the waterline inside the sewage system. This oxidizing process
prefers to take place where there is sufficiently high local temperature, enough
productions of hydrogen sulfide gas, high relative humidity, and atmospheric
oxygen; see section \ref{chemistry}  for more details on the involved chemistry
and transport mechanisms. Good overviews of the civil engineering literature on
the chemical aggression with acids of cement-based materials [focussing on
sulfate ingress] can be found in
\cite{BeddoeDorner2005,JahaniEtAl2001a,MullauerEtAl2009,TixierEtAl2003}.

If we decouple the mechanical corrosion part (leading to cracking and respective
spalling of the concrete matrix) from the reaction-diffusion-flow part, and look
only to the later one, the mathematical problem reduces to solving a partly
dissipative reaction-diffusion system posed in heterogeneous domains. Now,
assuming further that the concrete sample is perfectly covered by a
locally-periodic repeated regular microstructure, averaged and two-scale
reaction-diffusion systems modeling this corrosion processes can be derived;
that is precisely what we have done in \cite{FatimaEtAl2010} (formal asymptotics
for the locally-periodic case) and \cite{Tasnim2scale} (rigorous asymptotics via
two-scale convergence for the periodic case).

Here, our attention focusses on the two-scale corrosion model. Besides
performing the averaging procedure and ensuring the well-posedness of the
resulting model(s), we are interested in simulating numerically the influence of
the microstructural effects on observable (macroscopic) quantities. We refer the
reader to \cite{ChalupeckyEtAl2010}, where  we performed numerical simulations
of such a two-scale model. Now, is the right moment to raise the main question
of this paper:

\begin{center}
{\em Is the two-scale finite difference scheme used in \cite{ChalupeckyEtAl2010}
convergent, i.e., does it approximate the weak solution to the two-scale system?}
\end{center}

It is worth mentioning that there is a wealth of multiscale numerical techniques
that could (in principle) be used to tackle RD systems of the type treated here.
We mention at this point three approaches only: (i) the multiscale FEM method
developed by Babuska and predecessors (see the book \cite{Hou} for more Refs.),
(ii) computations on two-scale FEM spaces \cite{Matache} / two-scale Galerkin
approximations \cite{MunteanNeuss2010,MunteanLakkis2010}, and (iii) the
philosophy of heterogeneous multiscale methods (HMM) \cite{HMM}. We choose to
employ here multiscale finite differences (multiscale FD) mimicking the
[two-scale] tensorial structure present in (ii).  Our  hope is to become able to
marry at a later stage the two-scale Galerkin approximation ideas from
\cite{MunteanNeuss2010,Ludwig} in a HMM framework eventually based on finite
differences. Our standard reference for FD-HMM idea is \cite{AbdulleFD}.

The paper is structured as follows: Section \ref{background} introduces the
reader to the physico-chemical background of the corrosion process, two-scale
geometry, and setting of the equations. The numerical scheme together with  basic
ingredients like discrete operators, discrete Green formulae, discrete trace
inequalities etc.  are presented in section \ref{sec:numerical-scheme}, while the
approximation estimates together with the interpolation (extension) and
compactness steps are the subject of Section \ref{compactness} . We conclude the
paper with Section \ref{illustration} containing numerical illustrations of
the discrete approximation of the weak solution.

\section{Background and statement of the problem}\label{background}

\subsection{Two-scale geometry}\label{geometry}

We consider the evolution of a chemical corrosion process (sulfate attack)
taking place in one-dimensional macroscopic region $\Omega\coleq(0,L)$, $L>0$,
that represents a concrete sample along a line perpendicular to the pipe surface
with $x=0$ being a point at the inner surface in contact with sewer atmosphere
and $x=L$ being a point inside the concrete wall. Since we do not take into
account bulging of the inner surface due to the growth of soft gypsum
structures, the shape of the domain $\Omega$ does not change w.r.t. the time
variable $t$.

We denote the typical microstructure (or standard cell \cite{Hornung1997}) by
$Y:=(0,\ell)$, $\ell>0$. Usually cells in concrete contain a stationary water
film, and air and solid fractions in different ratios depending on the local
porosity. Generally, we expect that, due to the randomness of the pores
distribution in concrete, the choice of the microstructure essentially depends
on the macroscopic position $x\in\Omega$, i.e., we would then have $Y_x$; see
\cite{NoordenMuntean2010} for averaging issues of double porosity media
involving locally periodic ways of distributing microstructures, and
\cite{FatimaEtAl2010} for more comments directly related to the sulfatation
problem where pores are distributed in a locally periodic fashion. In this
paper, we restrict to the case when  the medium $\Omega$ is made of the {\em
same} microstructure $Y$ periodically repeated to pave perfectly the region.
Furthermore, since at the microscopic level the involved reaction and diffusion
processes take place in the pore water, we choose to denote by $Y$ only the wet
part of the pore. Efficient direct computations (with controlled accuracy and
known convergence rates) of scenarios involving $Y_x$ as well as the
corresponding error analysis are generally open problems in the field of
multiscale numerical simulation.

\subsection{Chemistry}\label{chemistry}

Sewage is rich in sulphur-containing materials and normally it is without any
action on concrete. Under suitable conditions like increased temperature or
lower flow velocity oxygen in sewage can become depleted. Aerobic, purifying
bacteria become inactive while anaerobic bacteria that live in slime layers at
the bottom of the sewer pipe proliferate. They obtain needed oxygen by reducing
sulfur compounds. Sulfur reacts with hydrogen and forms hydrogen sulfide (${\rm
H_2S}$), which then diffuses in sewage and enters sewer atmosphere. It moves in
the air space of the pipe and goes up towards the pipe wall. Gaseous ${\rm
H_2S}$ (further denoted as ${\rm H_2S(g)}$) enters into the concrete pores
(microstructures) via both air and water parts. ${\rm H_2S(g)}$ diffuses quickly
through the air-filled part of the porous structure over macroscopic distances,
while it dissolves in the thin, stationary water film of much smaller,
microscopic thickness that clings on the surface of the fabric.

There are many chemical reactions taking place in the porous microstructure of
sewer pipes which degrade the performance of the pipe structure depending on the
intensity of the interaction between the chemical reactions and the local
environment. Here we focus our attention on the following few relevant chemical
reactions:
\begin{subequations}\label{eq:reaction-equations}
  \begin{align}
    \label{eq:reaction-equations-1}
    {\rm H_2S(aq) + 2O_2} &\rightarrow {\rm 2H^+ + SO_4^{2-}} \\
    \label{eq:reaction-equations-2}
    {\rm 10H^{+} + SO_4^{-2} + organic\ matter} &\rightarrow {\rm H_2S(aq) +\ 4H_2O\ + oxidized\ matter} \\
    \label{eq:reaction-equations-3}
    {\rm H_2S(aq)} &\rightleftharpoons {\rm H_2S(g)} \\
    \label{eq:sulfatation}
    {\rm 2H_2O + H^+ + SO_4^{2-} + CaCO_3} &\rightarrow {\rm HCO_{3^-} + CaSO_4 \cdot 2H_2O.}
  \end{align}
\end{subequations}
Dissolved hydrogen sulfide (further denoted as ${\rm H_2S(aq)}$)
undergoes oxidation by aerobic bacteria living in these films and
sulfuric acid ${\rm H_2SO_4}$ is produced (reaction
\eqref{eq:reaction-equations-1}). This aggressive acid reacts with
calcium carbonate (i.e., with our concrete sample) and a soft gypsum
layer $({\rm CaSO_4\cdot 2H_2O})$ consisting of solid particles
(unreacted cement, aggregate), pore air and moisture is formed
(reaction \eqref{eq:sulfatation}).

The model considered in this paper pays special attention to the following
aspects:
\begin{itemize}
  \item[(i)] exchange of ${\rm H_2S}$ from water to the air phase and vice versa
    (reaction \eqref{eq:reaction-equations-3};
    %see \cite{BallsLiss1983} for more comments on mass transfer across air-liquid interfaces),
  \item[(ii)] production of gypsum at micro solid-water interfaces (reaction
    \eqref{eq:sulfatation}).
\end{itemize}
The transfer of ${\rm H_2S}$ is modeled by means of (deviations from) the Henry's law, while the
production of gypsum is incorporated in a non-standard non-linear reaction rate,
here denoted as $\eta$; see \eqref{eq:reaction-rate} for a precise choice.
Equation \eqref{eq:Henry} indicates the linearity of the Henry's law structure
we have in mind. The standard reference for modeling gas-liquid reactions at
stationary interfaces (including a derivation via first principles of the
Henry's law) is \cite{Danckwerts1970}.

\subsection{Setting of the equations}

Let $S\coleq(0,T)$ (with $T\in (0,\infty)$) be the time interval during which we
consider the process and let $\Omega$ and $Y$ as described in section
\ref{geometry}. We look for the unknown functions (mass concentrations of active
chemical species)
\begin{align*}
  u_1&:\Omega\times S \to \R &&\text{-- concentration of H$_2$S(g)},\\
  u_2&:\Omega\times Y\times S \to \R &&\text{-- concentration of H$_2$S(aq)},\\
  u_3&:\Omega\times Y\times S \to \R &&\text{-- concentration of H$_2$SO$_4$},\\
  u_4&:\Omega\times S \to \R &&\text{-- concentration of gypsum},
\end{align*}
that satisfy the following two-scale system composed of three weakly coupled
PDEs and one ODE
\begin{subequations}\label{eq:model}
\begin{align}
  \label{eq:model-1}
  \partial_t u_1 - d_1 \partial_{xx} u_1 &= d_2\partial_y u_2|_{y=0}, & \quad \text{
  in } & \Omega,\\
  \label{eq:model-2}
  \partial_t u_2 - d_2 \partial_{yy} u_2 &= -\zeta(u_2,u_3), & \quad
  \text{ in } & \Omega\times Y,\\
  \label{eq:model-3}
  \partial_t u_3 - d_3 \partial_{yy} u_3 &=  \zeta(u_2,u_3), & \quad
  \text{ in } & \Omega\times Y,\\
  \label{eq:model-4}
  \partial_t u_4 &= \eta( u_3|_{y=\ell}, u_4 ), & \quad \text{ in } & \Omega,
\end{align}
\end{subequations}
together with boundary conditions
\begin{subequations}\label{eq:model-bc}
  \begin{align}
  \label{eq:model-bc-1}
  u_1 &= u_1^D,& &\text{ on } \{x=0\}\times S,\\
  \label{eq:model-bc-2}
  d_1\partial_x u_1 &= 0,& &\text{ on } \{x=L\}\times S,\\
  \label{eq:model-bc-3}
  -d_2\partial_y u_2 &= Bi^M ( H u_1 - u_2 ),&  &\text{ on } \Omega\times \{y=0\}\times S,\\
  \label{eq:model-bc-4}
  d_2\partial_y u_2 &= 0,& &\text{ on } \Omega\times \{y=\ell\}\times S,\\
  \label{eq:model-bc-5}
  -d_3\partial_y u_3 &= 0,& &\text{ on } \Omega\times \{y=0\}\times S,\\
  \label{eq:model-bc-6}
  d_3\partial_y u_3 &= -\eta(u_3, u_4),& &\text{ on } \Omega\times
  \{y=\ell\}\times S,
  \end{align}
\end{subequations}
and initial conditions
\begin{equation}\label{eq:model-ic}
  \begin{aligned}
    u_1 &= u_1^0,& \text{ on } &\Omega\times\{t=0\},\\
    u_2 &= u_2^0,& \text{ on } &\Omega\times Y \times\{t=0\},\\
    u_3 &= u_3^0,& \text{ on } &\Omega\times Y \times\{t=0\},\\
    u_4 &= u_4^0,& \text{ on } &\Omega\times\{t=0\}.
  \end{aligned}
\end{equation}
Here $d_k$, $k\in\{1,2,3\}$, are the diffusion coefficients, $Bi^M$ is a
dimensionless Biot number, $H$ is the Henry's constant, $\alpha,\beta$ are
air-water mass transfer functions, while $\eta(\cdot)$ is a surface chemical
reaction. Note that $u_i$ ($i=1,\dots,4$) are mass concentrations. Furthermore,
all unknown functions, data and parameters carry dimensions.

\subsubsection{Technical assumptions}

The initial and boundary data, the parameters as well as the involved chemical
reaction rate are assumed to satisfy the following requirements:

 (A1) $d_k>0$, $k\in\{1,2,3\}$, $Bi^M>0$, $H>0$, $u_1^D>0$ are constants;
 % and $\alpha,\beta\inL^\infty_+(Y)$;

 (A2) The function $\zeta$ represents the biological oxidation volume reaction
 between the hydrogen sulfide and sulfuric acid and is defined by
\begin{align}\label{eq:zeta}
\zeta:\mathbb{R}^2\to\mathbb{R}, \  \zeta(r,s) &\coleq \alpha r -
\beta s,
\end{align}
where $\alpha,\beta\in L_+^\infty(Y)$.

(A3)  We assume the reaction rate $\eta:\mathbb{R}^2\to\mathbb{R}_+$
takes the form
\begin{equation}
  \label{eq:reaction-rate}
  \eta(r,s) =
  \begin{cases}
    kR(r)Q(s),& \text{ for all } r\geq 0, s\geq 0,\\
    0, &\text{ otherwise},
  \end{cases}
\end{equation}
where $k>0$ is the corresponding reaction constant. We assume $\eta$ to be
(globally) Lipschitz in both arguments. Furthermore, $R$ is taken to be
sublinear (i.e., $R(r)\leq r$ for all $r\in \mathbb{R}$, in the spirit of
\cite{Berz}), while $Q$ is bounded from above by a threshold $\bar c>0$.
Furthermore, let $R\in W^{1,\infty}(0,M_3)$ and $Q\in W^{1,\infty}(0,M_4)$ be
monotone functions (with $R$ strictly increasing), where the constants $M_3$ and
$M_4$ are the $L^\infty$ bounds on $u_3$ and, respectively, on $u_4$. Note that
\cite[Lemma~2]{ChalupeckyEtAl2010} gives the constants $M_3,M_4$ explicitly.

(A4)
$u_1^0\in H^2(\Omega)\cap L^\infty_+(\Omega)$,
$(u_2^0,u_3^0)\in\left[L^2(\Omega;H^1(Y))\right]^2\times \left[L^\infty_+(\Omega\times Y)\right]^2$,
$u_4^0\in H^1(\Omega)\cap L^\infty_+(\Omega)$;

\subsubsection{Micro-macro transmission}

Terms like
\begin{equation}\label{eq:Henry}
  Bi^M\big( H u_1(x,t) - u_2(x,y=0,t)\big)
\end{equation}
are usually referred to in the mathematical literature as production terms by
Henry's or Raoult's law; see \cite{BoehmEtAl1998}. The special feature of our
scenario is that the term \eqref{eq:Henry} {\em bridges} two distinct spatial
scales: one macro with $x\in \Omega$ and one micro with $y\in Y$. We call this
{\em micro-macro transmission condition}.

It is important to note that in the subsequent analysis we can replace
(\ref{eq:Henry}) by a more general nonlinear relationship
$$\mathcal{B}(u_1,u_2).$$
In that case assumption (A2) needs to be replaced, for instance, by
(A2')
\begin{align}\label{eq:zeta2}
  \mathcal{B}\in C^1([0,M_1]\times [0,M_2];\mathbb{R}),
  & \  \mathcal{B}\mbox{ globally Lipschitz in both arguments },
\end{align}
where $M_1$ and $M_2$ are sufficiently large positive constants\footnote{Typical
choices for $M_1,M_2$ are the $L^\infty$-estimates on $u_1$ and $u_2$; cf.
\cite{ChalupeckyEtAl2010} (Lemma 2) such $M_1,M_2$ do exist.}. Note that a
derivation of the precise structure of $\mathcal{B}$ by taking into account
(eventually by averaging of) the underlying microstructure information is still
an open problem.

\subsection{Weak formulation}

As a next step, we first reformulate our problem \eqref{eq:model},
\eqref{eq:model-bc}, \eqref{eq:model-ic} in an equivalent formulation that is
more suitable for numerical treatment. We introduce the substitution
$\tilde{u}_1\coleq u_1 - u_1^D$ to obtain
\begin{subequations}\label{eq:ref-model}
\begin{align}
  \label{eq:ref-model-1}
  \partial_t \tilde{u}_1 - d_1 \partial_{xx} \tilde{u}_1 &= d_2\partial_y u_2|_{y=0}, & \quad \text{
  in } & \Omega,\\
  \label{eq:ref-model-2}
  \partial_t u_2 - d_2 \partial_{yy} u_2 &= -\zeta(u_2,u_3), & \quad
  \text{ in } & \Omega\times Y,\\
  \label{eq:ref-model-3}
  \partial_t u_3 - d_3 \partial_{yy} u_3 &=  \zeta(u_2,u_3), & \quad
  \text{ in } & \Omega\times Y,\\
  \label{eq:ref-model-4}
  \partial_t u_4 &= \eta( u_3|_{y=\ell}, u_4 ), & \quad \text{ in } & \Omega,
\end{align}
\end{subequations}
together with boundary conditions
\begin{subequations}\label{eq:ref-model-bc}
  \begin{align}
  \label{eq:ref-model-bc-1}
  \tilde{u}_1 &= 0,& &\text{ on } \{x=0\}\times S,\\
  \label{eq:ref-model-bc-2}
  d_1\partial_x \tilde{u}_1 &= 0,& &\text{ on } \{x=L\}\times S,\\
  \label{eq:ref-model-bc-3}
  -d_2\partial_y u_2 &= Bi^M \big( H (\tilde{u}_1 + u_1^D) - u_2 \big),&  &\text{ on } \Omega\times \{y=0\}\times S,\\
  \label{eq:ref-model-bc-4}
  d_2\partial_y u_2 &= 0,& &\text{ on } \Omega\times \{y=\ell\}\times S,\\
  \label{eq:ref-model-bc-5}
  -d_3\partial_y u_3 &= 0,& &\text{ on } \Omega\times \{y=0\}\times S,\\
  \label{eq:ref-model-bc-6}
  d_3\partial_y u_3 &= -\eta(u_3, u_4),& &\text{ on } \Omega\times
  \{y=\ell\}\times S,
  \end{align}
\end{subequations}
and initial conditions
\begin{equation}\label{eq:ref-model-ic}
  \begin{aligned}
    \tilde{u}_1 &= u_1^0 - u_1^D \equalscolon \tilde{u}_1^0,& \text{ on } &\Omega\times\{t=0\},\\
    u_2 &= u_2^0,& \text{ on } &\Omega\times Y \times\{t=0\},\\
    u_3 &= u_3^0,& \text{ on } &\Omega\times Y \times\{t=0\},\\
    u_4 &= u_4^0,& \text{ on } &\Omega\times\{t=0\}.
  \end{aligned}
\end{equation}
We  refer to the system \eqref{eq:ref-model}, \eqref{eq:ref-model-bc},
\eqref{eq:ref-model-ic} as problem (P). Also, for the ease of notation, we
denote $\tilde{u}_1$ again as $u_1$ and $\tilde{u}_1^0$ as $u_1^0$.

Now, we can introduce our concept of weak solution.
\begin{defn}[Concept of weak solution]\label{def:weak}
  The vector of functions
  $(u_1,u_2,u_3,u_4)$
  with
  \begin{align}
    \label{eq:weak-spaces-1}
    u_1 &\in L^2(S,H_0^1(\Omega)),\\
    \label{eq:weak-spaces-2}
    \partial_t u_1 &\in L^2(S\times\Omega),\\
    \label{eq:weak-spaces-3}
    u_i &\in L^2(S,L^2(\Omega,H^1(Y))),& i&\in\{2,3\},\\
    \label{eq:weak-spaces-4}
    \partial_t u_i &\in L^2(S\times\Omega\times Y),& i&\in\{2,3\},\\
    \label{eq:weak-spaces-5}
    u_4(\cdot,x,y) &\in H^1(S),& \text{for a.e. }(x,y)&\in \Omega\times Y,
  \end{align}
  is called {\em a weak solution} to problem (P) if the identities
  \begin{align*}
    \int_\Omega\partial_t u_1\varphi_1 + d_1\int_\Omega\partial_x u_1\partial_x\varphi_1 &= \int_\Omega\partial_y u_2|_{y=0}\varphi_1,\\
    \int_\Omega\int_Y \partial_t u_2\varphi_2 + d_2\int_\Omega\int_Y\partial_y
    u_2\partial_y\varphi_2 &= -\int_\Omega\int_Y \zeta(u_2,u_3)\varphi_2 -
    \int_\Omega\partial_y u_2|_{y=0}\varphi_2,\\
    \int_\Omega\int_Y \partial_t u_3\varphi_3 + d_3\int_\Omega\int_Y\partial_y
    u_3\partial_y\varphi_3 &= \int_\Omega\int_Y \zeta(u_2,u_3)\varphi_3 -
    \int_\Omega\eta(u_3|_{y=\ell},u_4)\varphi_3,
  \end{align*}
  and
  \begin{equation*}
    \partial_t u_4 = \eta(u_3|_{y=\ell},u_4),
  \end{equation*}
  hold for a.e. $t\in S$ and for all
  $\varphi\coleq(\varphi_1,\varphi_2,\varphi_3)\in
  H_0^1(\Omega) \times \left[L^2(\Omega;H^1(Y))\right]^2$.
\end{defn}

We refer the reader to \cite[Theorem~3]{ChalupeckyEtAl2010} for statements
regarding the global existence and uniqueness of such weak solutions to problem
(P); see also \cite{MunteanNeuss2010} for the  analysis on a closely related
problem.

The main question we are dealing with here is:

\begin{center}
{\em How to approximate the weak solution in an easy and efficient way,
consistent with the structure of the model and the regularity of the data and
parameters indicated in (A1)--(A4)?}
\end{center}

%%%%%%%%%%%%%%%%%%%%%%%%%%%%%%%%%%%%%%%%%%%%%%%%%%%%%%%%%%%%%%%%%%%%%%%%%%%%%%%%
\section{Numerical scheme}\label{sec:numerical-scheme}

In order to solve numerically our multiscale system
\eqref{eq:model}--\eqref{eq:model-ic}, we use a semi-discrete approach leaving
the time variable continuous and discretizing both space variables $x$ and $y$
by finite differences on rectangular grids. In the following paragraphs we
introduce the necessary notation, the scheme and discrete scalar products and
norms.

%%%%%%%%%%%%%%%%%%%%%%%%%%%%%%%%%%%%%%%%%%%%%%%%%%%%%%%%%%%%%%%%%%%%%%%%%%%%%%%%
\subsection{Grids and grid functions}\label{sec:grids}

For spatial discretization, we subdivide the domain $\Omega$ into $N_x$
equidistant subintervals, the domain $Y$ into $N_y$ equidistant subintervals and
we denote by $h_x\coleq L/N_x$, $h_y\coleq\ell/N_y$, the corresponding spatial
step sizes. We denote by $h$ the vector $(h_x,h_y)$ with length $|h|$.

Let
\begin{align*}
  \Omegah &\coleq \{ x_i \coleq ih_x \ |\ i=0,\dotsc,N_x \},\\
  \Omegaho &\coleq \{ x_i \ |\ i=1,\dotsc,N_x \},\\
  Y_h &\coleq \{ y_j \coleq jh_y \ |\ i=0,\dotsc,N_y \},\\
  \Omegahe &\coleq \{ x_{i+1/2}\coleq (i+1/2)h_x \ |\ i=0,\dotsc,N_x-1 \},\\
  Y^e_h &\coleq \{ y_{j+1/2} \coleq (j+1/2)h_y \ |\ i=0,\dotsc,N_y-1 \},
\end{align*}
be, respectively, grid of all nodes in $\Omega$, grid of nodes in $\Omega$
without the node at $x=0$ (where Dirichlet boundary condition will be imposed),
grid of all nodes in $Y$, grid of nodes located in the middle of subintervals of
$\Omegah$, and grid of nodes located in the middle of subintervals of $Y_h$.
Finally, we define grids $\omega_h\coleq \Omegah\times Y_h$ and
$\omegahe\coleq\Omegah\times Y^e_h$.

Next, we introduce grid functions defined on the grids just described. Let $\Gh
\coleq \{ u_h\ |\ u_h:\Omegah\to\R \}$, $\Gho \coleq \{ u_h\ |\
u_h:\Omegaho\to\R \}$and $\Eh \coleq \{ \mrm{v}_h | \mrm{v}_h:\Omegahe\to\R \}$
be sets of grid functions approximating macro variables on $\Omega$. Let $\Fh
\coleq \{ u_h\ |\ u_h:\omega_h\to\R \}$ and $\Hh \coleq \{ \mrm{v}_h\ |\
\mrm{v}_h:\omegahe\to\R \}$ be sets of grid functions approximating micro
variables on $\Omega\times Y$. These grid functions can be identified with
vectors in $\R^N$, whose elements are the values of the grid function at the
nodes of the respective grid. Hence, addition of functions and multiplication of
a function by a scalar are defined as for vectors.

For $u_h\in\Gh$ we denote $u_i\coleq u_h(x_i)$, and for $u_h\in\Fh$ we will
denote $u_{ij}\coleq u_h(x_i,y_j)$. For $\mrm{v}_h\in\Eh$ we will denote
$\mrm{v}_{i+1/2}\coleq\mrm{v}_h(x_{i+1/2})$, and for $\mrm{v}_h\in\Hh$ we will
denote $\mrm{v}_{i,j+1/2}\coleq\mrm{v}_h(x_i,y_{j+1/2})$.

We frequently use functions from $\Fh$ restricted to the sets
$\Omega_h\times\{y=0\}$ or $\Omega_h\times\{y=\ell\}$. For $u_h\in\Fh$, we will
denote these restrictions as $u_h|_{y=0}$ and $u_h|_{y=\ell}$, and we will
interpret them as functions from $\Gh$, i.e., $u_h|_{y=0}\in\Gh$ and
$u_h|_{y=\ell}\in\Gh$.

%%%%%%%%%%%%%%%%%%%%%%%%%%%%%%%%%%%%%%%%%%%%%%%%%%%%%%%%%%%%%%%%%%%%%%%%%%%%%%%%
\subsection{Discrete operators}

In this section, we define difference operators defined on  linear spaces of
grid functions in such a way they mimic properties of the corresponding
differential operators and, together with the scalar products defined in Sec.
\ref{sec:scalar-products}, fulfill similar integral identities.

The discrete gradient operators $\nabla_h$ and $\nabla_{yh}$ are defined as
\begin{align*}
\nabla_h &: \Gh\to\Eh,& \qquad (\nabla_h u_h)_{i+\frac{1}{2}}&\coleq\frac{u_{i+1}-u_i}{h_x}, \quad u_h\in\Gh,\\
\nabla_{yh} &: \Fh\to\Hh,& \qquad (\nabla_{yh}
u_h)_{i,j+\frac{1}{2}}&\coleq\frac{u_{i,j+1}-u_{ij}}{h_y}, \quad
u_h\in\Fh,
\end{align*}
while the discrete divergence operators $\Div_h$ and $\Div_{yh}$ is
\begin{align*}
\Div_h &: \Eh\to\Gho,& \qquad (\Div_h \mrm{v}_h)_{i}&\coleq\frac{\mrm{v}_{i+\frac{1}{2}}-\mrm{v}_{i-\frac{1}{2}}}{h_x}, \quad \mrm{v}_h\in\Eh,\\
\Div_{yh} &: \Hh\to\Fh,& \qquad (\Div_{yh}
\mrm{v}_h)_{ij}&\coleq\frac{\mrm{v}_{i,j+\frac{1}{2}}-\mrm{v}_{i,j-\frac{1}{2}}}{h_y},
\quad \mrm{v}_h\in\Hh.
\end{align*}
The discrete Laplacian operators $\Delta_h$ and $\Delta_{yh}$ are defined as
$\Delta_h \coleq \Div_h\nabla_h:\Gh\to\Gho$ and $\Delta_{yh}\coleq
\Div_{yh}\nabla_{yh}:\Fh\to\Fh$, i.e., the following standard 3-point stencils
are obtained:
\begin{align*}
  (\Delta_h u_h)_{i} &= \frac{u_{i-1}-2u_i+u_{i+1}}{h_x^2},\qquad u_h\in\Gh,\\
  (\Delta_{yh} u_h)_{ij} &= \frac{u_{i,j-1}-2u_{ij}+u_{i,j+1}}{h_y^2},\qquad u_h\in\Fh.
\end{align*}

To complete the definition of the discrete divergence and Laplacian operators,
we need to specify values of grid functions on auxiliary nodes that fall outside
their corresponding grid. At a later point, we obtain these values  from the
discretization of boundary conditions by centered differences.

%%%%%%%%%%%%%%%%%%%%%%%%%%%%%%%%%%%%%%%%%%%%%%%%%%%%%%%%%%%%%%%%%%%%%%%%%%%%%%%%
\subsection{Semi-discrete scheme}

We can now construct a semi-discrete scheme for problem \eqref{eq:model}. Note
that we omit the explicit dependence on $t$ and we interchangeably use the
notation $\deriv{u_h}{t}$ and $\dot{u}_h$ for denoting the derivative of $u_h$
with respect to $t$.

\begin{defn}\label{def:scheme}
  A quadruple $\{u^1_h,u^2_h,u^3_h,u^4_h\}$ with
  $$u^1_h,u^4_h\in C^1([0,T];\Gh)\mbox{ and }u^2_h,u^3_h\in C^1([0,T];\Fh)$$
  is called  semi-discrete
  solution of \eqref{eq:model}, if it satisfies the following system of ordinary
  differential equations
  \begin{subequations}\label{eq:scheme}
    \begin{align}
      \label{eq:scheme-1}
      \deriv{u^1_h}{t} &= d_1 \Delta_h u^1_h - Bi^M\big( H(u^1_h+u_1^D) - u^2_h|_{y=0} \big),&&\text{ on } \Omegaho,\\
      \label{eq:scheme-2}
      \deriv{u^2_h}{t} &= d_2 \Delta_{yh} u^2_h - \zeta( u^2_h, u^3_h ),&&\text{ on } \omega_h,\\
      \label{eq:scheme-3}
      \deriv{u^3_h}{t} &= d_3 \Delta_{yh} u^3_h + \zeta( u^2_h, u^3_h ),&&\text{ on } \omega_h,\\
      \label{eq:scheme-4}
      \deriv{u^4_h}{t} &= \eta(u^3_h|_{y=\ell},u^4_h),&&\text{ on } \Omegah,
    \end{align}
  \end{subequations}
  together with the discrete boundary conditions ($i=0,\dotsc,N_x$)
  \begin{subequations}\label{eq:scheme-bc}
    \begin{align}
      \label{eq:scheme-bc-1}
      u^1_0 &= 0,\\
      \label{eq:scheme-bc-2}
      d_1\frac{1}{2}\Big( (\nabla_h u^1_h)_{N_x+\frac{1}{2}} + (\nabla_h u^1_h)_{N_x-\frac{1}{2}} \Big) &= 0,\\
      \label{eq:scheme-bc-3}
      -d_2\frac{1}{2}\Big( (\nabla_{yh} u^2_h)_{i,-\frac{1}{2}} + (\nabla_{yh} u^2_h)_{i,\frac{1}{2}} \Big) &= Bi^M\big( H(u^1_i + u_1^D) - u^2_{i,0} \big),\\
      \label{eq:scheme-bc-4}
      d_2\frac{1}{2}\Big( (\nabla_{yh} u^2_h)_{i,N_y+\frac{1}{2}} + (\nabla_{yh} u^2_h)_{i,N_y-\frac{1}{2}} \Big) &= 0,\\
      \label{eq:scheme-bc-5}
      -d_3\frac{1}{2}\Big( (\nabla_{yh} u^3_h)_{i,-\frac{1}{2}} + (\nabla_{yh} u^3_h)_{i,\frac{1}{2}} \Big) &= 0,\\
      \label{eq:scheme-bc-6}
      d_3\frac{1}{2}\Big( (\nabla_{yh} u^3_h)_{i,N_y+\frac{1}{2}} + (\nabla_{yh}
      u^3_h)_{i,N_y-\frac{1}{2}} \Big) &= -\eta( u^3_{i,N_y}, u^4_i),
    \end{align}
  \end{subequations}
  and the initial conditions
  \begin{equation}\label{eq:scheme-ic}
    \begin{aligned}
      u^1_h(0) &= \PhO u_1^0,& u^2_h(0) &= \Pho u_2^0,\\
      u^3_h(0) &= \Pho u_3^0,& u^4_h(0) &= \PhO u_4^0,
    \end{aligned}
  \end{equation}
  where $\PhO$ and $\Pho$ are suitable projection operators from $\Omega$ to
  $\Omegah$ and from $\Omega\times Y$ to $\omegah$, respectively.
\end{defn}

\begin{rem}
  The boundary conditions \eqref{eq:scheme-bc-2}--\eqref{eq:scheme-bc-6} are a
  centered-difference approximation of conditions \eqref{eq:model-bc} and are
  written so as to stress the relation between the two. Using the definition of
  discrete $\nabla_h$ and $\nabla_{yh}$ operators, \eqref{eq:scheme-bc} can be
  rewritten in terms of auxiliary values of $u^k_h$, $k=1,\dotsc,4$, on nodes
  outside the grids as follows:
  \begin{subequations}\label{eq:scheme-bcs}
    \begin{align}
      \label{eq:scheme-bcs-1}
      u^1_0 &= 0,\\
      \label{eq:scheme-bcs-2}
      u^1_{N_x+1} &= u^1_{N_x-1},\\
      \label{eq:scheme-bcs-3}
      u^2_{i,-1} &= u^2_{i,1} + \frac{2h_y}{d_2} Bi^M\big( H(u^1_i + u_1^D) - u^2_{i,0} \big),\\
      \label{eq:scheme-bcs-4}
      u^2_{i,N_y+1} &= u^2_{i,N_y-1},\\
      \label{eq:scheme-bcs-5}
      u^3_{i,-1} &= u^3_{i,1},\\
      \label{eq:scheme-bcs-6}
      u^3_{i,N_y+1} &= u^3_{i,N_y-1} - \frac{2h_y}{d_3} \eta( u^3_{i,N_y}, u^4_i).
    \end{align}
  \end{subequations}
\end{rem}

\begin{prop}
  Assume (A1)--(A4) to be fulfilled. Then there exists a unique semi-discrete
  solution $$\{u^1_h,u^2_h,u^3_h,u^4_h\}\in C^1([0,T];\Gh)\times
  C^1([0,T];\Fh)\times C^1([0,T];\Gh)\times C^1([0,T];\Fh)$$ in the sense of
  Definition \ref{def:scheme}.
\end{prop}
\begin{proof}
  The proof, based on the standard ode argument, follows in a straightforward
  manner.
\end{proof}

%%%%%%%%%%%%%%%%%%%%%%%%%%%%%%%%%%%%%%%%%%%%%%%%%%%%%%%%%%%%%%%%%%%%%%%%%%%%%%%%
\subsection{Discrete scalar products and norms}\label{sec:scalar-products}

Next, we introduce scalar products and norms on the spaces of grid
functions $\Gh$, $\Eh$, $\Fh$, $\Gh$ and we show some basic integral
identities for the difference operators.

Let $(\gamma_i^1)_{i=0}^{N_x}$ and $(\gamma_j^2)_{j=0}^{N_y}$ be such that
\begin{equation}\label{eq:scalar-prod-coef}
    \gamma^1_{i} \coleq
    \begin{cases}
      1 & 1\leq i \leq N_x-1,\\
      \frac{1}{2} & i\in\{0,N_x\},
    \end{cases}, \qquad
    \gamma^2_{j} \coleq
    \begin{cases}
      1 & 1\leq j \leq N_y-1,\\
      \frac{1}{2} & j\in\{0,N_y\},
    \end{cases}
\end{equation}
and define the following discrete $L^2$ scalar products and the corresponding
discrete $L^2$ norms
\begin{align}
  \label{eq:sp-Gh}
  (u_h,v_h)_{\Gh} &\coleq h_x \sum_{x_{i}\in\Omegah}
  \gamma_i^1 u_{i} v_{i},&\qquad u_h, v_h &\in \Gh,\\
  \label{eq:norm-Gh}
  \| u_h \|_{\Gh} &\coleq \sqrt{ (u_h,u_h)_{\Gh} },&\qquad u_h&\in\Gh,
\end{align}
\begin{align}
  \label{eq:sp-Gho}
  (u_h,v_h)_{\Gho} &\coleq h_x \sum_{x_{i}\in\Omegaho}
  \gamma_i^1 u_{i} v_{i},&\qquad u_h, v_h &\in \Gho,\\
  \label{eq:norm-Gho}
  \| u_h \|_{\Gho} &\coleq \sqrt{ (u_h,u_h)_{\Gho} },&\qquad u_h&\in\Gho,
\end{align}
\begin{align}
  \label{eq:sp-Fh}
  (u_h,v_h)_{\Fh} &\coleq h_x h_y \sum_{x_{ij}\in\omega_h} \gamma_i^1\gamma_j^2 u_{ij} v_{ij},&\qquad u_h, v_h &\in \Fh,\\
  \label{eq:norm-Fh}
  \| u_h \|_{\Fh} &\coleq \sqrt{ (u_h,u_h)_{\Fh} },& \qquad u_h&\in\Fh,
\end{align}
\begin{align}
  \label{eq:sp-Eh}
  (\mrm{u}_h,\mrm{v}_h)_{\Eh} &\coleq h_x \sum_{x_{i+1/2}\in\Omegahe}
  \mrm{u}_{i+1/2} \mrm{v}_{i+1/2},&\qquad \mrm{u}_h, \mrm{v}_h &\in \Eh,\\
  \label{eq:norm-Eh}
  \| \mrm{u}_h \|_{\Eh} &\coleq \sqrt{ (\mrm{u}_h,\mrm{u}_h)_{\Eh} },& \qquad
  \mrm{u}_h&\in\Eh,
\end{align}
\begin{align}
  \label{eq:sp-Hh}
  (\mrm{u}_h,\mrm{v}_h)_{\Hh} &\coleq h_x h_y \sum_{x_{i,j+1/2}\in\omegahe}
  \gamma^1_i\mrm{u}_{i,j+1/2} \mrm{v}_{i,j+1/2},&\qquad \mrm{u}_h, \mrm{v}_h &\in \Hh,\\
  \label{eq:norm-Hh}
  \| \mrm{u}_h \|_{\Hh} &\coleq \sqrt{ (\mrm{u}_h,\mrm{u}_h)_{\Hh} },& \qquad
  \mrm{u}_h&\in\Hh.
\end{align}

It can be shown that a discrete equivalent of Green's formula holds for these
scalar products as well as other identities as is stated in the following
lemmas.

\begin{lem}[Discrete macro Green-like formula]\label{lem:sp-props-1}
  Let $u_h\in\Gh$ and $\mrm{v}_h\in\Eh$ such that
  \begin{gather}
    \label{eq:sp-Gh-cond-1}
    u_{0} = 0,\quad u_{N_x+1} = u_{N_x-1},\\
    \label{eq:sp-Gh-cond-2}
    \mrm{v}_{N_x+1/2} = -\mrm{v}_{N_x-1/2}.
  \end{gather}
  Then the following identity holds:
  \begin{equation}
    \label{eq:sp-Gh-prop-1}
    (u_h, \Div_h\mrm{v}_h)_{\Gho} = - (\nabla_h u_h, \mrm{v}_h )_{\Eh}.
    % \label{eq:sp-Gh-prop-2}
    % (u_h, \Delta_h u_h)_{\Gho} &= - (\nabla_h u_h, \nabla_h u_h )_{\Eh}.
  \end{equation}
\end{lem}

\begin{lem}[Discrete micro-macro Green-like formula]\label{lem:sp-props-2}
  Let $u_h\in\Fh$ and $\mrm{v}_h\in\Hh$ such that
  \begin{gather}
    \label{eq:sp-Fh-cond-1}
    -\frac{1}{2}\left(\mrm{v}_{k,-1/2} + \mrm{v}_{k,1/2}\right) = \delta^1_k,\quad
    \frac{1}{2}\left(\mrm{v}_{k,N_y-1/2} + \mrm{v}_{k,N_y+1/2}\right) = \delta^2_k,\\
    \label{eq:sp-Fh-cond-2}
    u_{k,-1} = u_{k,1} + 2h_y\delta^1_k,\qquad u_{k,N_y+1} = u_{k,N_y-1} + 2h_y\delta^2_k,
  \end{gather}
  for $i=0,\dotsc,N_x$, and $\delta^1_h,\delta^2_h\in\Gh$. Then the following
  identity holds:
  \begin{equation}
    \label{eq:sp-Fh-prop-1}
    (u_h, \Div_{yh}\mrm{v}_h)_{\Fh} = - (\nabla_{yh} u_h, \mrm{v}_h)_{\Hh} +
    (u_h|_{y=0}, \delta^1_h)_{\Gh} + (u_h|_{y=N_y}, \delta^2_h)_{\Gh}.
    % \label{eq:sp-Fh-prop-2}
    % (u_h, \Delta_{yh} u_h)_{\Fh} &= - (\nabla_{yh} u_h, \nabla_{yh} u_h )_{\Hh}
    % + ( u_h|_{y=0}, \delta^1_h )_{\Gh} + ( u_h|_{y=N_y}, \delta^2_h)_{\Gh}.
  \end{equation}
\end{lem}

We also frequently make use of the following discrete trace inequality:

\begin{lem}[Discrete trace inequality]\label{lem:discrete-trace}
  For $u_h\in\Fh$ there exists a positive constant $C$ depending only on
  $\Omega$ such that
  \begin{equation} \label{eq:discrete-trace}
    \|u_h|_{y=\ell}\|_{\Gh} \leq C(\|u_h\|_{\Fh}+\|\nabla_{yh}u_h\|_{\Hh}).
  \end{equation}
\end{lem}
\begin{proof}
  Our proof follows the line of thought  of  \cite{GallouetEtAl2000}. We have that for $u_h\in\Fh$
  \[
    |u_{i,N_y}| \leq \sum_{j=0}^{N_y-1}|u_{i,j+1}-u_{ij}| +
    \sum_{j=0}^{N_y}\gamma_j^2 h_y |u_{ij}|.
  \]
  Squaring both sides of the inequality, we get
  \begin{equation} \label{eq:discrete-trace-1}
    (u_{i,N_y})^2 \leq A_i + B_i,
  \end{equation}
  where
  \begin{equation*}
    A_i \coleq 2\left( \sum_{j=0}^{N_y-1}|u_{i,j+1}-u_{ij}| \right)^2\quad\mbox{ and }
    B_i \coleq 2\left( \sum_{j=0}^{N_y}\gamma_j^2 h_y |u_{ij}| \right)^2.
  \end{equation*}
  Applying the Cauchy-Schwarz inequality to $A_i$, we obtain
  \begin{equation*}
    A_i \leq 2 \sum_{j=0}^{N_y-1} h_y\left(\frac{u_{i,j+1}-u_{ij}}{h_y}\right)^2
    \sum_{j=0}^{N_y-1} h_y = 2\ell \sum_{j=0}^{N_y-1}
    h_y\left(\frac{u_{i,j+1}-u_{ij}}{h_y}\right)^2.
  \end{equation*}
  Similarly, using the Cauchy-Schwarz inequality we get for $B_i$
  \begin{equation*}
    B_i \leq 2 \sum_{j=0}^{N_y}\gamma_j^2 h_y (u_{ij})^2
    \sum_{j=0}^{N_y}\gamma_j^2 h_y = 2\ell\sum_{j=0}^{N_y}\gamma_j^2 h_y
    (u_{ij})^2.
  \end{equation*}
  Multiplying \eqref{eq:discrete-trace-1} by $\gamma_i^1 h_x$, summing over
  $i\in\{0,\dotsc,N_x\}$ and then using the bounds on $A_i$ and $B_i$,
  it yields that:
  \begin{equation*} \label{eq:discrete-trace-2}
    \sum_{i=0}^{N_x}\gamma_i^1h_x(u_{i,N_y})^2
    \leq
    2\ell\Bigg( \sum_{i=0}^{N_x}\sum_{j=0}^{N_y-1} \gamma_i^1 h_x h_y \left(
    (\nabla_{yh} u_h)_{i,j+\frac{1}{2}}\right)^2
    +
    \sum_{x_{ij}\in\omegah} \gamma_i^1\gamma_j^2h_x h_y (u_{ij})^2\Bigg),
  \end{equation*}
  that is
  \begin{equation*} \label{eq:discrete-trace-3}
    \|u_h|_{y=\ell}\|_{\Gh}^2 \leq C \left( \|\nabla_{yh}u_h\|_{\Hh}^2 +
    \|u_h\|_{\Fh}^2 \right),
  \end{equation*}
  from which the claim of the Lemma  follows directly.
\end{proof}

%%%%%%%%%%%%%%%%%%%%%%%%%%%%%%%%%%%%%%%%%%%%%%%%%%%%%%%%%%%%%%%%%%%%%%%%%%%%%%%%
\section{Approximation estimates}\label{estimates}

The aim of this section is to derive {\em a~priori} estimates on the
semi-discrete solution. Based on weak convergence-type arguments, the estimates
will ensure, at least up to subsequences, a (weakly) convergent way to
reconstruct the weak solution to problem (P).

%%%%%%%%%%%%%%%%%%%%%%%%%%%%%%%%%%%%%%%%%%%%%%%%%%%%%%%%%%%%%%%%%%%%%%%%%%%%%%%%
\subsection{{\em A~priori} estimates}

This is the place where we use the tools developed in section
\ref{sec:numerical-scheme}.

In subsequent paragraphs, we refer to the following relations: From
scalar product of \eqref{eq:scheme-1} with $\vph^1\in\Gh$,
\eqref{eq:scheme-2} and \eqref{eq:scheme-3} with $\vph^2\in\Fh$ and
$\vph^3$, respectively, and \eqref{eq:scheme-4} with $\vph^4\in\Gh$
to obtain
\begin{align}
  \label{eq:discrete-weak-1}
  (\dot{u}^1_h,\vph^1)_{\Gho} &= d_1 (\Delta_h u^1_h,\vph^1)_{\Gho} - Bi^M \big( H u^1_h - u^2_h|_{y=0},\vph^1\big)_{\Gho},\\
  \label{eq:discrete-weak-2}
  (\dot{u}^2_h,\vph^2)_{\Fh} &= d_2 (\Delta_{yh} u^2_h,\vph^2)_{\Fh} - \alpha (u^2_h,\vph^2)_{\Fh} + \beta (u^3_h,\vph^2)_{\Fh},\\
  \label{eq:discrete-weak-3}
  (\dot{u}^3_h,\vph^3)_{\Fh} &= d_3 (\Delta_{yh} u^3_h,\vph^3)_{\Fh} + \alpha (u^2_h,\vph^3)_{\Fh} - \beta (u^3_h,\vph^3)_{\Fh},\\
  \label{eq:discrete-weak-4}
  (\dot{u}^4_h,\vph^4)_{\Gh} &= \big(\eta(u^3_h|_{y=\ell},u^4_h), \vph^4\big)_{\Gh}.
\end{align}
Note that $u^1_h$ and $\nabla_h\vph^1$ satisfy the assumptions of Lemma
\ref{lem:sp-props-1}, $u^2_h$ and $\nabla_{yh}\vph^2$ satisfy the assumptions of
Lemma \ref{lem:sp-props-2} with $\delta^1_k=\frac{Bi^M}{d_2}(Hu^1_k-u^2_{k,0})$
and $\delta^2_k=0$, and $u^3_h$ and $\nabla_{yh}\vph^3$ with $\delta^1_k=0$ and
$\delta^2_k=-\frac{1}{d_3}\eta(u^3_{k,N_y},u^4_k)$. Thus, using Lemmas
\ref{lem:sp-props-1}, \ref{lem:sp-props-2} and properties of the discrete scalar
products we get
\begin{align}
  \label{eq:discrete-weak-mod-1}
  (\dot{u}^1_h,\vph^1)_{\Gh} + d_1 (\nabla_h u^1_h,\nabla_h \vph^1)_{\Eh} &= - Bi^M \big(Hu^1_h - u^2_h|_{y=0}, \vph^1\big)_{\Gh},\\
  \label{eq:discrete-weak-mod-2}
  (\dot{u}^2_h,\vph^2)_{\Fh} + d_2 (\nabla_{yh} u^2_h,\nabla_{yh} \vph^2)_{\Hh}
  &= Bi^M (Hu^1_h - u^2_h|_{y=0}, \vph^2|_{y=0})_{\Gh} - \alpha(u^2_h,\vph^2)_{\Fh} + \beta(u^3_h,\vph^2)_{\Fh},\\
  \label{eq:discrete-weak-mod-3}
  (\dot{u}^3_h,\vph^3)_{\Fh} + d_3 (\nabla_{yh}u^3_h, \nabla_{yh}\vph^3)_{\Hh} &=
    - \big(\eta(u^3_h|_{y=\ell},u^4_h),\vph^3|_{y=\ell}\big)_{\Gh}
    + \alpha (u^2_h,\vph^3)_{\Fh} - \beta (u^3_h,\vph^3)_{\Fh},\\
  \label{eq:discrete-weak-mod-4}
  (\dot{u}^4_h,\vph^4)_{\Gh} &= \big(\eta(u^3_h|_{y=\ell},u^4_h), \vph^4\big)_{\Gh}.
\end{align}

\begin{lem}[Discrete energy estimates] \label{lem:discrete-energy-estimate}
  Let $\{u^1_h,u^2_h,u^3_h,u^4_h\}$ be a semi-discrete solution of
  \eqref{eq:model} for some $T>0$. Then it holds that
  \begin{align}
    \max_{t\in S}\Big( \|u^1_h(t)\|_{\Gh}^2 + \|u^2_h(t)\|_{\Fh}^2 +
    \|u^3_h(t)\|_{\Fh}^2 + \|u^4_h(t)\|_{\Gh}^2 \Big)
    & \leq C,
    \label{eq:apriori-1} \\
    \int_0^T\Big( \|\nabla_h u^1_h\|_{\Eh}^2 + \|\nabla_{yh} u^2_h\|_{\Hh}^2 + \|\nabla_{yh} u^3_h\|_{\Hh}^2 \Big) \dt
    & \leq C,
    \label{eq:apriori-2}
  \end{align}
  where $C \coleq \bar{C} \Big( \|u^1_h(0)\|_{\Gh}^2 + \|u^2_h(0)\|_{\Fh}^2 + \|u^3_h(0)\|_{\Fh}^2 + \|u^4_h(0)\|_{\Gh}^2 \Big)$,
  with $\bar{C}$ being a positive constant independent of $h_x$, $h_y$.
\end{lem}
\begin{proof}
  In \eqref{eq:discrete-weak-mod-1}--\eqref{eq:discrete-weak-mod-4}, taking
  $(\vph^1,\vph^2,\vph^3,\vph^4)=(u^1_h,u^2_h,u^3_h,u^4_h)$,
  % \begin{align*}
    % \frac{1}{2}\deriv{}{t}\|u^1_h\|_{\Gh}^2 + d_1 \|\nabla_h u^1_h\|_{\Eh}^2 &= - Bi^M ( H u^1_h - u^2_h|_{y=0},u^1_h)_{\Gh},\\
    % \frac{1}{2}\deriv{}{t}\|u^2_h\|_{\Fh}^2 + d_2 \|\nabla_{yh} u^2_h\|_{\Hh}^2 &= Bi^M ( H u^1_h - u^2_h|_{y=0}, u^2_h|_{y=0} )_{\Gh} - \alpha \|u^2_h\|_{\Fh}^2 \nonumber\\
    % &\qquad\qquad + \beta (u^3_h,u^2_h)_{\Fh},\\
    % \frac{1}{2}\deriv{}{t}\|u^3_h\|_{\Fh}^2 + d_3 \|\nabla_{yh} u^3_h\|_{\Hh}^2
    % &= - \big(\eta(u^3_h|_{y=\ell},u^4_h),u^3_h|_{y=\ell}\big)_{\Gh} + \alpha (u^2_h,u^3_h)_{\Fh} \nonumber\\
    % &\qquad\qquad - \beta \|u^3_h\|_{\Fh}^2,\\
    % \frac{1}{2}\deriv{}{t}\|u^4_h\|_{\Gh}^2 &= \big(\eta(u^3_h|_{y=\ell},u^4_h), u^4_h\big)_{\Gh}.
  % \end{align*}
  summing the equalities, applying Young's inequality on terms with
  $(u^2_h,u^3_h)_{\Fh}$, dropping the negative terms on the right-hand side, and
  multiplying the resulting inequality by $2$ give
  \begin{multline*}
    \deriv{}{t} \Big( \|u^1_h\|_{\Gh}^2 + \|u^2_h\|_{\Fh}^2 +
    \|u^3_h\|_{\Fh}^2 + \|u^4_h\|_{\Gh}^2 \Big)
    + 2d_1 \|\nabla_h u^1_h\|_{\Eh}^2
    + 2d_2 \|\nabla_{yh} u^2_h\|_{\Hh}^2
    + 2d_3 \|\nabla_{yh} u^3_h\|_{\Hh}^2 \\
    \leq
    - 2Bi^M ( H u^1_h - u^2_h|_{y=0}, u^1_h )_{\Gh}
    + 2Bi^M ( H u^1_h - u^2_h|_{y=0}, u^2_h|_{y=0} )_{\Gh} \\
    + C_1\|u^2_h\|_{\Fh}^2
    + C_1\|u^3_h\|_{\Fh}^2
    + 2\big(\eta(u^3_h|_{y=\ell},u^4_h), u^4_h - u^3_h|_{y=\ell} \big)_{\Gh},
  \end{multline*}
  where $C_1\coleq \alpha+\beta>0$. Expanding the first two terms on the
  right-hand side we get
  \begin{multline*}
    - 2( H u^1_h - u^2_h|_{y=0}, u^1_h )_{\Gh} + 2( H u^1_h -
    u^2_h|_{y=0}, u^2_h|_{y=0} )_{\Gh} = - 2H \|u^1_h\|_{\Gh}^2 \\ 
    +2(1+H)(u^1_h,u^2_h|_{y=0})_{\Gh} - 2\|u^2_h|_{y=0}\|_{\Gh}^2 \leq
    \frac{1+H}{\varepsilon}\|u^1_h\|_{\Gh}^2 +
    \big((1+H)\varepsilon - 2\big) \|u^2_h|_{y=0}\|_{\Gh}^2,
  \end{multline*}
  where we used Young's inequality with $\varepsilon>0$. Choosing $\varepsilon$
  sufficiently small, the coefficient in front of the last term is negative, so
  we have that
  \[
    - 2(Hu^1_h-u^2_h|_{y=0},u^1_h)_{\Gh} + 2(Hu^1_h-u^2_h|_{y=0},u^2_h|_{y=0})_{\Gh}
    \leq C_2\|u^1_h\|_{\Gh}^2,
  \]
  where $C_2\coleq \frac{1+H}{\varepsilon}>0$, and thus
  \begin{multline}\label{eq:apriori-proof-1}
    \deriv{}{t} \Big( \|u^1_h\|_{\Gh}^2 + \|u^2_h\|_{\Fh}^2 +
    \|u^3_h\|_{\Fh}^2 + \|u^4_h\|_{\Gh}^2 \Big)
    + 2d_1 \|\nabla_h u^1_h\|_{\Eh}^2
    + 2d_2 \|\nabla_{yh} u^2_h\|_{\Hh}^2
    + 2d_3 \|\nabla_{yh} u^3_h\|_{\Hh}^2 \\
    \leq
    C_2\|u^1_h\|_{\Gh}^2
    + C_1\|u^2_h\|_{\Fh}^2
    + C_1\|u^3_h\|_{\Fh}^2
    + 2\big(\eta(u^3_h|_{y=\ell},u^4_h), u^4_h - u^3_h|_{y=\ell} \big)_{\Gh}.
  \end{multline}
  For the last term on the right-hand side of the previous inequality we have
  \begin{multline*}
    2\big(\eta(u^3_h|_{y=\ell},u^4_h), u^4_h - u^3_h|_{y=\ell} \big)_{\Gh} =
    2k\big(R(u^3_h|_{y=\ell})Q(u^4_h), u^4_h\big)_{\Gh}
    \underbrace{- 2k\big(R(u^3_h|_{y=\ell})Q(u^4_h),u^3_h|_{y=\ell}
    \big)_{\Gh}}_{\leq 0}\\
    \leq 2k\bar{c}^q\big(R(u^3_h|_{y=\ell}), u^4_h\big)_{\Gh}
    \leq k\bar{c}^q\Big( \varepsilon\|R(u^3_h|_{y=\ell})\|_{\Gh}^2 + \frac{1}{\varepsilon}\|u^4_h\|_{\Gh}^2 \Big)
    \leq k\bar{c}^q\Big( \varepsilon\|u^3_h|_{y=\ell}\|_{\Gh}^2 + \frac{1}{\varepsilon}\|u^4_h\|_{\Gh}^2 \Big) \\
    \leq k\bar{c}^q\Big(
    C_3\varepsilon\|u_h^3\|_{\Fh}^2 + C_3\varepsilon\|\nabla_{yh}u_h^3\|_{\Hh}^2
    + \frac{1}{\varepsilon}\|u^4_h\|_{\Gh}^2 \Big),
  \end{multline*}
  where we used the assumption (A$_1$), Young's inequality with $\varepsilon>0$
  and the discrete trace inequality \eqref{eq:discrete-trace} with the constant
  $C_3>0$. Using the result in \eqref{eq:apriori-proof-1} we obtain
  \begin{multline}\label{eq:apriori-proof-2}
    \deriv{}{t} \Big( \|u^1_h\|_{\Gh}^2 + \|u^2_h\|_{\Fh}^2 +
    \|u^3_h\|_{\Fh}^2 + \|u^4_h\|_{\Gh}^2 \Big)
    + d_1 \|\nabla_h u^1_h\|_{\Eh}^2
    + d_2 \|\nabla_{yh} u^2_h\|_{\Hh}^2
    + C_4 \|\nabla_{yh} u^3_h\|_{\Hh}^2 \\
    \leq
    C_2\|u^1_h\|_{\Gh}^2 + C_1\|u^2_h\|_{\Fh}^2 + C_5\|u_h^3\|_{\Fh}^2 + C_6\|u^4_h\|_{\Gh}^2,
  \end{multline}
  where
  $C_4\coleq d_3 - k\bar{c}^qC_3\varepsilon$ can be made positive for
  $\varepsilon$ sufficiently small, $C_5\coleq C_1 +
  k\bar{c}^qC_3\varepsilon$, and $C_6\coleq
  k\bar{c}^q\frac{1}{\varepsilon}$.

  Discarding the terms with discrete gradient, we get
  \begin{equation*}
    \deriv{}{t}
    \Big(  \|u^1_h\|_{\Gh}^2 + \|u^2_h\|_{\Fh}^2 + \|u^3_h\|_{\Fh}^2 + \|u^4_h\|_{\Gh}^2 \Big)
    \leq C_7\Big( \|u^1_h\|_{\Gh}^2 + \|u^2_h\|_{\Fh}^2
    + \|u^3_h\|_{\Fh}^2 + \|u^4_h\|_{\Gh}^2 \Big),
  \end{equation*}
  where $C_7\coleq \max\{C_1,C_2,C_5,C_6\}$. Applying the Gronwall's lemma to
  the previous inequality we obtain
  \begin{multline}\label{eq:apriori-proof-3}
    \max_{t\in S}\Big( \|u^1_h(t)\|_{\Gh}^2 + \|u^2_h(t)\|_{\Fh}^2 + \|u^3_h(t)\|_{\Fh}^2 +
    \|u^4_h(t)\|_{\Gh}^2 \Big) \\
    \leq
    \Big( \|u^1_h(0)\|_{\Gh}^2 + \|u^2_h(0)\|_{\Fh}^2 + \|u^3_h(0)\|_{\Fh}^2 +
    \|u^4_h(0)\|_{\Gh}^2 \Big) e^{C_7T}.
  \end{multline}
  Finally, integrating \eqref{eq:apriori-proof-2} over $[0,T]$ and using
  \eqref{eq:apriori-proof-3} gives
  \begin{multline}\label{eq:apriori-proof-4}
    \int_0^T\Big( \|\nabla_h u^1_h\|_{\Eh}^2 + \|\nabla_{yh} u^2_h\|_{\Hh}^2 + \|\nabla_{yh} u^3_h\|_{\Hh}^2 \Big) \dt \\
    \leq
    \frac{1 + C_7 T e^{C_7T}}{C_8} \Big( \|u^1_h(0)\|_{\Gh}^2 +
    \|u^2_h(0)\|_{\Fh}^2 + \|u^3_h(0)\|_{\Fh}^2 + \|u^4_h(0)\|_{\Gh}^2 \Big),
  \end{multline}
  where $C_8 \coleq \min\{d_1,d_2,C_4\}$. The claim of the lemma directly
  follows.
\end{proof}

\begin{lem} \label{lem:time-derivative-estimate}
  Let $\{u^1_h,u^2_h,u^3_h,u^4_h\}$ be a semi-discrete solution of
  \eqref{eq:model} for some $T>0$. Then it holds that
  \begin{align}
    \label{eq:apriori-timeder-1}
    \max_{t\in S}\Big( \|\dot{u}^1_h(t)\|_{\Gh}^2 + \|\dot{u}^2_h(t)\|_{\Fh}^2 +
    \|\dot{u}^3_h(t)\|_{\Fh}^2 \Big) & \leq C,\\
    \label{eq:apriori-timeder-2}
    \int_0^T\Big(\|\nabla_h\dot{u}^1_h\|_{\Eh}^2 + \|\nabla_{yh}\dot{u}^2_h\|_{\Hh}^2 + \|\nabla_{yh}\dot{u}^3_h\|_{\Hh}^2\Big) \dt
    & \leq C,
  \end{align}
  where $C$ is a positive constant independent of $h_x$, $h_y$.
\end{lem}
\begin{proof}
  We follow the steps of \cite[Theorem~4]{MunteanNeuss2010}. Differentiate
  \eqref{eq:discrete-weak-mod-1}--\eqref{eq:discrete-weak-mod-3} with respect to
  time, take $\vph^i=\dot{u}^i_h,i=1,\dotsc,3$, discard the negative terms on
  the right-hand side and sum the inequalities to obtain
  \begin{multline*}
      \frac{1}{2}\deriv{}{t}\Big(
      \|\dot{u}^1_h\|^2_{\Gh} +
      \|\dot{u}^2_h\|^2_{\Fh} +
      \|\dot{u}^3_h\|^2_{\Fh}
      \Big)
      + d_1\|\nabla_h\dot{u}^1_h\|^2_{\Eh}
      + d_2\|\nabla_{yh}\dot{u}^2_h\|^2_{\Hh}
      + d_3\|\nabla_{yh}\dot{u}^3_h\|^2_{\Hh} \\
      \leq
      Bi^M (1 + H) \big(\dot{u}^1_h, \dot{u}^2_h|_{y=0}\big)_{\Gh} - Bi^M\|\dot{u}^2_h|_{y=0}\|^2_{\Gh}
      + (\alpha+\beta)\big(\dot{u}^2_h,\dot{u}^3_h\big)_{\Fh} \\
      - \big(\partial_r\eta(u^3_h|_{y=\ell},u^4_h)\dot{u}^3_h|_{y=\ell} +
      \partial_s\eta(u^3_h|_{y=\ell},u^4_h)\dot{u}^4_h, \dot{u}_h^3|_{y=\ell} \big)_{\Gh}.
  \end{multline*}
  As in the proof of Lemma \ref{lem:discrete-energy-estimate}, for the first
  two terms on the right-hand side we have that
  \begin{equation*}
      Bi^M (1 + H) \big(\dot{u}^1_h, \dot{u}^2_h|_{y=0}\big)_{\Gh} -
      Bi^M\|\dot{u}^2_h|_{y=0}\|^2_{\Gh} \leq C_1 \|\dot{u}^1_h\|^2_{\Gh},
  \end{equation*}
  and for the third term
  \begin{equation*}
      (\alpha+\beta)\big(\dot{u}^2_h,\dot{u}^3_h\big)_{\Fh} \leq
      C_2 \big( \|\dot{u}^2_h\|^2_{\Fh} + \|\dot{u}^3_h\|^2_{\Fh} \big).
  \end{equation*}
  Using the Lipschitz property of $\eta$, together with Schwarz's  and Young's
  inequalities, and assuming the structural restriction $\partial_r\eta > 0$, we
  obtain for the last term on the right-hand side that
  \begin{multline*}
      - \big(\partial_r\eta\dot{u}^3_h|_{y=\ell} + \partial_s\eta\dot{u}^4_h,
      \dot{u}_h^3|_{y=\ell} \big)_{\Gh} =
      -\big(\partial_r\eta\dot{u}^3_h|_{y=\ell},\dot{u}_h^3|_{y=\ell}\big)_{\Gh}
      - \big(\partial_s\eta\dot{u}^4_h, \dot{u}_h^3|_{y=\ell} \big)_{\Gh} \\
      \leq
      \underbrace{-\big(\partial_r\eta\dot{u}^3_h|_{y=\ell},\dot{u}_h^3|_{y=\ell}\big)_{\Gh}}_{\leq 0}
      + C \left(\frac{1}{2\varepsilon} \big\|\dot{u}^4_h\big\|^2_{\Gh} +
      \frac{\varepsilon}{2}\|\dot{u}_h^3|_{y=\ell} \big\|^2_{\Gh}\right).
  \end{multline*}
  Choosing $\varepsilon$ sufficiently small, we get that
  \begin{equation*}
      - \big(\partial_r\eta(u^3_h|_{y=\ell},u^4_h)\dot{u}^3_h|_{y=\ell} +
      \partial_s\eta(u^3_h|_{y=\ell},u^4_h)\dot{u}^4_h, \dot{u}_h^3|_{y=\ell}
      \big)_{\Gh} \leq C_3\big\|\dot{u}^4_h\big\|^2_{\Gh}.
  \end{equation*}
  Putting the obtained results together we finally obtain that
  \begin{multline}\label{eq:time-derivative-estimate-1}
      \frac{1}{2}\deriv{}{t}\Big(
      \|\dot{u}^1_h\|^2_{\Gh} +
      \|\dot{u}^2_h\|^2_{\Fh} +
      \|\dot{u}^3_h\|^2_{\Fh}
      \Big)
      + d_1\|\nabla_h\dot{u}^1_h\|^2_{\Eh}
      + d_2\|\nabla_{yh}\dot{u}^2_h\|^2_{\Hh}
      + d_3\|\nabla_{yh}\dot{u}^3_h\|^2_{\Hh} \\
      \leq
      C_1 \|\dot{u}^1_h\|^2_{\Gh}
      + C_2 \big(\|\dot{u}^2_h\|^2_{\Fh} + \|\dot{u}^3_h\|^2_{\Fh} \big)
      + C_3 \|\dot{u}^4_h\|^2_{\Gh}.
  \end{multline}
  Gr\"onwall's inequality gives that
  \begin{equation}\label{eq:time-derivative-estimate-2}
    \max_{t\in S}\Big( \|\dot{u}^1_h\|^2_{\Gh} + \|\dot{u}^2_h\|^2_{\Fh} +
    \|\dot{u}^3_h\|^2_{\Fh} \Big) \leq C_4 \Big( \|\dot{u}^1_h(0)\|^2_{\Gh} + \|\dot{u}^2_h(0)\|^2_{\Fh} +
    \|\dot{u}^3_h(0)\|^2_{\Fh} \Big).
  \end{equation}
  In order to estimate the right-hand side in the previous inequality, we evaluate
  \eqref{eq:discrete-weak-1}--\eqref{eq:discrete-weak-3} at $t=0$ and test with
  $\big(\dot{u}^1_h(0),\dot{u}^2_h(0),\dot{u}^3_h(0)\big)$ to get
  \begin{multline*}
  \|\dot{u}^1_h(0)\|^2_{\Gh} + \|\dot{u}^2_h(0)\|^2_{\Fh} + \|\dot{u}^3_h(0)\|^2_{\Fh}
    = d_1 (\Delta_h u^1_h(0),\dot{u}^1_h(0))_{\Gho} + d_2 (\Delta_{yh}
    u^2_h(0),\dot{u}^2_h(0))_{\Fh} \\ + d_3 (\Delta_{yh} u^3_h(0),\dot{u}^3_h(0))_{\Fh}
    - Bi^M \big( H u^1_h(0) - u^2_h(0)|_{y=0},\dot{u}^1_h(0)\big)_{\Gh} \\
    + (\alpha u^2_h(0) - \beta u^3_h(0),\dot{u}^3_h(0)-\dot{u}^2_h(0))_{\Fh}.
  \end{multline*}
  Schwarz's inequality and Young's inequality (with $\varepsilon>0$ chosen
  sufficiently small) together with the regularity of the initial data yield the
  estimate
  \begin{equation*}
    \|\dot{u}^1_h(0)\|^2_{\Gh} + \|\dot{u}^2_h(0)\|^2_{\Fh} +
    \|\dot{u}^3_h(0)\|^2_{\Fh} \leq C,
  \end{equation*}
  where $C$ does not depend on the spatial step sizes. Returning back to
  \eqref{eq:time-derivative-estimate-1}, integrating it with respect to $t$ and
  using \eqref{eq:time-derivative-estimate-2} gives the claim of the lemma.
\end{proof}

In the following lemma we derive additional {\em a~priori} estimates that will
finally allow us to pass in the limit in the non-linear terms. In order to avoid
introducing new grids, grid functions and associated scalar products for finite
differences in $x$ variable, we will resort to sum notation in this proof. To
this end, for $u_h\in\Fh$, let $\delta^+_x u_{ij}$, $\delta^-_x u_{ij}$,
$\delta^+_y u_{ij}$, $\delta^-_y u_{ij}$ denote the forward and backward
difference quotients at $x_{ij}$ in $x$- and $y$-direction, i.e.,
\begin{gather*}
  (\delta^+_x u_h)_{ij}\coleq\frac{u_{i+1,j}-u_{ij}}{h_x},\quad
  (\delta^-_x u_h)_{ij}\coleq\frac{u_{ij}-u_{i-1,j}}{h_x},\\
  (\delta^+_y u_h)_{ij}\coleq\frac{u_{i,j+1}-u_{ij}}{h_y},\quad
  (\delta^-_y u_h)_{ij}\coleq\frac{u_{ij}-u_{i,j-1}}{h_y}.
\end{gather*}
\begin{lem}[Improved {\em a~priori} estimates]\label{lem:improved-apriori}
  Let $\{u^1_h,u^2_h,u^3_h,u^4_h\}$ be a semi-discrete solution of
  \eqref{eq:model} for some $T>0$. Then it holds that
  \begin{align}
    \max_{t\in S}\Big(
    h_xh_y\sum_{i=0}^{N_x-1}\sum_{j=0}^{N_y}(\delta_x^+u^2_{ij})^2
    +
    h_xh_y\sum_{i=0}^{N_x-1}\sum_{j=0}^{N_y}(\delta_x^+u^3_{ij})^2
    \Big) &\leq C,
    \\
    \int_0^T h_xh_y\sum_{i=0}^{N_x-1}\sum_{j=0}^{N_y-1}(\delta^+_x\delta^+_yu^2_{ij})^2 \,\dt
    +
    \int_0^T h_xh_y\sum_{i=0}^{N_x-1}\sum_{j=0}^{N_y-1}(\delta_x^+\delta_y^+u^3_{ij})^2 \,\dt
    &\leq C,
  \end{align}
  where $C$ is a positive constant independent of $h_x$, $h_y$.
\end{lem}
\begin{proof}
  Following the steps of \cite[Theorem~5]{MunteanNeuss2010}, introduce a
  function $\vartheta\in C^\infty_0(\Omega)$ such that $0\leq\vartheta\leq 1$
  and let $\vartheta_h\coleq\vartheta|_{\Omega_h}\in\Gh$. Test
  \eqref{eq:scheme-2} with $-\delta_x^-(\vartheta^2_i\delta_x^+u^2_h)_{ij}$,
  \eqref{eq:scheme-3} with $-\delta_x^-(\vartheta^2_i\delta_x^+u^3_h)_{ij}$, and
  sum over $\omega_h$ to form relations analogous to \eqref{eq:discrete-weak-mod-2}, \eqref{eq:discrete-weak-mod-3}. We get
  \begin{multline*}
    -h_y\sum_{i=1}^{N_x}\sum_{j=0}^{N_y}\gamma^1_i\gamma^2_j \dot{u}^2_{ij}\delta_x^-(\vartheta_h^2\delta_x^+ u^2_h)_{ij}
    -d_2h_y\sum_{i=1}^{N_x}\sum_{j=0}^{N_y-1}\gamma^1_i\gamma^2_j \delta^+_y u^2_{ij} \delta^+_y(\delta^-_x(\vartheta_h^2\delta_x^+u^2_h))_{ij}\\
    = -Bi^m \sum_{i=1}^{N_x}\gamma^1_i (Hu^1_{i}-u^2_{i,0})\delta_x^-(\vartheta^2_i\delta_x^+ u^2_h)_{i,0}
    + \alpha h_y\sum_{i=1}^{N_x}\sum_{j=0}^{N_y}\gamma^1_i\gamma^2_j u^2_{ij}\delta_x^-(\vartheta^2_i\delta_x^+ u^2_h)_{ij}\\
    - \beta  h_y\sum_{i=1}^{N_x}\sum_{j=0}^{N_y}\gamma^1_i\gamma^2_j u^3_{ij}\delta_x^-(\vartheta^2_i\delta_x^+ u^2_h)_{ij},
  \end{multline*}
  \begin{multline*}
    -h_y\sum_{i=1}^{N_x}\sum_{j=0}^{N_y}\gamma^1_i\gamma^2_j \dot{u}^3_{ij}\delta_x^-(\vartheta_h^2\delta_x^+ u^3_h)_{ij}
    -d_3h_y\sum_{i=1}^{N_x}\sum_{j=0}^{N_y-1}\gamma^1_i\gamma^2_j \delta^+_y u^3_{ij} \delta^+_y(\delta^-_x(\vartheta_h^2\delta_x^+u^3_h))_{ij}\\
    =
    \sum_{i=1}^{N_x}\gamma^1_i \eta(u^3_{i,N_y},u^4_i)\delta_x^-(\vartheta^2_i\delta_x^+ u^3_h)_{i,N_y}
    - \alpha h_y\sum_{i=1}^{N_x}\sum_{j=0}^{N_y}\gamma^1_i\gamma^2_j u^2_{ij}\delta_x^-(\vartheta^2_i\delta_x^+ u^3_h)_{ij}\\
    + \beta  h_y\sum_{i=1}^{N_x}\sum_{j=0}^{N_y}\gamma^1_i\gamma^2_j u^3_{ij}\delta_x^-(\vartheta^2_i\delta_x^+ u^3_h)_{ij}.
  \end{multline*}
  Summing the previous two equalities and using the discrete Green's theorem analogous
  to \eqref{eq:sp-Gh-prop-1}, Schwarz's inequality and Young's inequality we obtain
  \begin{multline}\label{eq:improved-apriori-proof-1}
    \frac{1}{2}\deriv{}{t}\Big(
    h_y\sum_{i=0}^{N_x-1}\sum_{j=0}^{N_y}|\vartheta_i\delta_x^+u^2_{ij}|^2
    + h_y\sum_{i=0}^{N_x-1}\sum_{j=0}^{N_y}|\vartheta_i\delta_x^+u^3_{ij}|^2
    \Big)
    + d_2h_y\sum_{i=0}^{N_x-1}\sum_{j=0}^{N_y-1}|\vartheta_i\delta^+_x\delta^+_yu^2_{ij}|^2\\
    + d_3h_y\sum_{i=0}^{N_x-1}\sum_{j=0}^{N_y-1}|\vartheta_i\delta^+_x\delta^+_yu^3_{ij}|^2
    \leq Bi^m H C_1 \sum_{i=0}^{N_x-1}|\vartheta_i\delta_x^+u^1_{i}|^2
    +
    C_2h_y\sum_{i=0}^{N_x-1}\sum_{j=0}^{N_y}(\vartheta_i\delta_x^+u^2_{ij})^2\\
    + C_3h_y\sum_{i=0}^{N_x-1}\sum_{j=0}^{N_y}(\vartheta_i\delta_x^+u^3_{ij})^2
    - \sum_{i=0}^{N_x-1}(\delta_x^+\eta(u^3_{i,N_y},u^4_i))(\vartheta_i^2\delta_x^+u^3_{i,N_y}).
  \end{multline}
  We rewrite the last term on the right-hand side as
  \begin{multline*}
    -k\sum_{i=0}^{N_x-1}(\delta_x^+(R(u^3_{i,N_y})Q(u^4_i)))(\vartheta_i^2\delta_x^+u^3_{i,N_y})
    \\
    = -k\sum_{i=0}^{N_x-1}\Big(Q(u^4_i)\delta_x^+R(u^3_{i,N_y}) + R(u^3_{i+1,N_y})\delta_x^+Q(u^4_i)\Big)(\vartheta_i^2\delta_x^+u^3_{i,N_y})
    \\
    = \underbrace{-k\sum_{i=0}^{N_x-1}\vartheta_i^2Q(u^4_i)\delta_x^+R(u^3_{i,N_y})\delta_x^+u^3_{i,N_y}}_{\leq
    0}-k\sum_{i=0}^{N_x-1}R(u^3_{i+1,N_y})\delta_x^+Q(u^4_i)(\vartheta_i^2\delta_x^+u^3_{i,N_y}),
  \end{multline*}
  where we used the monotonicity of $R$ and boundedness of $Q$. To estimate the
  last term we exploit the Lipschitz continuity and boundedness of $Q$ and use
  the discrete trace theorem so that
  \begin{multline*}
    -k\sum_{i=0}^{N_x-1}R(u^3_{i+1,N_y})\delta_x^+Q(u^4_i)(\vartheta_i^2\delta_x^+u^3_{i,N_y})
    \leq
    C_4\sum_{i=0}^{N_x-1}\vartheta_i^2|\delta_x^+u^3_{i,N_y}\delta_x^+u^4_i| \\
    \leq
    \frac{C_4\varepsilon}{2}\sum_{i=0}^{N_x-1}(\vartheta_i\delta_x^+u^3_{i,N_y})^2
    +\frac{C_4}{2\varepsilon}\sum_{i=0}^{N_x-1}(\vartheta_i\delta_x^+u^4_i)^2
    \leq
    C_5\varepsilon h_y\sum_{i=0}^{N_x-1}\sum_{j=0}^{N_y}(\vartheta_i\delta_x^+u^3_{ij})^2
    \\
    +C_5\varepsilon h_y\sum_{i=0}^{N_x-1}\sum_{j=0}^{N_y-1}(\vartheta_i\delta_x^+\delta_y^+u^3_{ij})^2
    +\frac{C_4}{2\varepsilon}\sum_{i=0}^{N_x-1}(\vartheta_i\delta_x^+u^4_i)^2.
  \end{multline*}
  Using the latter result in \eqref{eq:improved-apriori-proof-1}, we arrive at
  \begin{multline}\label{eq:improved-apriori-proof-2}
    \frac{1}{2}\deriv{}{t}\Big(
    h_y\sum_{i=0}^{N_x-1}\sum_{j=0}^{N_y}(\vartheta_i\delta_x^+u^2_{ij})^2
    + h_y\sum_{i=0}^{N_x-1}\sum_{j=0}^{N_y}(\vartheta_i\delta_x^+u^3_{ij})^2
    \Big)
    +d_2h_y\sum_{i=0}^{N_x-1}\sum_{j=0}^{N_y-1}(\vartheta_i\delta^+_x\delta^+_yu^2_{ij})^2 \\
    +(d_3 - C_5\varepsilon) h_y \sum_{i=0}^{N_x-1}\sum_{j=0}^{N_y-1}(\vartheta_i\delta_x^+\delta_y^+u^3_{ij})^2
    \leq
    Bi^m H C_1 \sum_{i=0}^{N_x-1}|\vartheta_i\delta_x^+u^1_{i}|^2
    + C_2h_y\sum_{i=0}^{N_x-1}\sum_{j=0}^{N_y}(\vartheta_i\delta_x^+u^2_{ij})^2
    \\
    + (C_3+C_5\varepsilon)h_y\sum_{i=0}^{N_x-1}\sum_{j=0}^{N_y}(\vartheta_i\delta_x^+u^3_{ij})^2
    +\frac{C_4}{2\varepsilon}\sum_{i=0}^{N_x-1}(\vartheta_i\delta_x^+u^4_i)^2.
  \end{multline}
  Applying Gronwall's inequality and integrating with respect to time we obtain
  the claim of the lemma.
\end{proof}

% \begin{lem}[Discrete Gronwall's lemma in general sum form]\label{lem:ac-gronwall}
    % Let $\{a_n\},\{b_n\} \subset \R$, $\lambda>0$, $\tau>0$ and
    % $\theta\in [0,1]$ such that $1-\theta\lambda\tau>0$. Then the
    % inequality
    % \[
    % a_n \leq b_n + \lambda \sum_{j=1}^n \tau \big( (1-\theta) a_{j-1} + \theta
    % a_j \big), \quad \forall n\in\N
    % \]
    % with $a_1 \leq b_1$ implies
    % \[
    % a_n \leq b_n + \frac{\lambda \tau}{1-\theta \lambda \tau}
    % \sum_{j=1}^n \left( \frac{1+(1-\theta)\lambda \tau}{1-\theta\lambda\tau}
    % \right)^{n-j}\big( (1-\theta)b_{j-1} + \theta b_j\big), \quad \forall
    % n\in\N.
    % \]
% \end{lem}

% \begin{rem}
    % If $\{b_n\} \equiv b$ is constant, then the claim of the previous lemma can
    % be simplified to
    % \begin{equation}\label{eq:ac-gronwall-simple}
        % a_n \leq b \left( \frac{1+(1-\theta)\lambda\tau}{1-\theta\lambda\tau}
        % \right)^n, \quad \forall n\in\N.
    % \end{equation}
% \end{rem}

\section{Interpolation and compactness}\label{compactness}

In this section, we derive sufficient results that enable us to show the
convergence of semi-discrete solutions of \eqref{eq:model}. To this end, we
firstly introduce extensions of grid functions so that they are defined almost
everywhere in $\Omega$ and $\omega$ and can be studied by the usual techniques
of Lebesgue/Sobolev/Bochner spaces. Finally, we use the {\em a~priori} estimates
proved in section \ref{estimates} to show the necessary compactness for the
sequences of extended grid functions.

\subsection{Interpolation}

In this subsection we introduce extensions of grid functions so that they are
defined almost everywhere in $\Omega$ and $\omega$.

\begin{defn}[Dual and simplicial grids on $\Omega$]
  Let $\Omegah$ be a grid on $\Omega$ as defined in Section \ref{sec:grids}.
  Define the dual grid $\OmegahD$ as
  \[
  \OmegahD \coleq \{ \KD_{i} \subset \bar{\Omega} \ |\ \KD_i\coleq [ x_i-h_x/2, x_i+h_x/2 ]\cap\bar{\Omega},\ x_i\in\Omega_h \},
  \]
  and the simplicial grid $\OmegahS$ as
  \[
  \OmegahS \coleq \{ \KS_{i} \subset \bar{\Omega} \ |\ \KS_i\coleq [ x_i, x_{i+1} ]\cap\bar{\Omega},\ x_i\in\Omega_h \}.
  \]
\end{defn}

\begin{defn}[Dual and simplicial grids on $\Omega\times Y$]
  Let $\omegah$ be a grid on $\Omega\times Y$ as defined in Section \ref{sec:grids}.
  Define the dual grid $\omegahD$ as
  \begin{multline*}
    \omegahD \coleq \{ \LD_{ij} \subset \bar{\Omega}\times\bar{Y} \ |\
    \LD_{ij}\coleq [ x_i-h_x/2, x_i+h_x/2 ]\\
    \times[y_j-h_y/2,
    y_j+h_y/2]\cap\bar{\Omega}\times\bar{Y},\ x_i\in\Omega_h, y_j\in Y_h \},
  \end{multline*}
  and the simplicial grid $\omegahS$ as
  $\omegahS\coleq\omegah^{\triangledl}\cup\omegah^{\triangleur}$, where
  \begin{align*}
    \omegah^{\triangledl}&\coleq \big\{ \LSDL_{ij}\ |\ \LSDL_{ij} \coleq
    \big[(x_i,y_j),(x_{i+1},y_j),(x_i,y_{j+1})\big]_{\varkappa}\cap\bar{\Omega}\times\bar{Y},
     i=0,\dotsc,N_x-1,j=0,\dotsc,N_y-1 \big\},\\
    \omegah^{\triangleur}&\coleq \big\{ \LSUR_{ij}\ |\ \LSUR_{ij} \coleq
    \big[(x_{i+1},y_{j+1}),(x_{i+1},y_j),(x_i,y_{j+1})\big]_{\varkappa}\cap\bar{\Omega}\times\bar{Y},
    i=0,\dotsc,N_x-1,j=0,\dotsc,N_y-1 \big\},
  \end{align*}
  where $[\mrm{x},\mrm{y},\mrm{z}]_{\varkappa}$ denotes convex hull of points
  $\mrm{x},\mrm{y},\mrm{z}\in\R^2$.
\end{defn}

\begin{defn}[Piecewise constant extension]
  For a grid function $u_h$ we define its piecewise constant extension
  $\bar{u}_h$ as
  \begin{equation}\label{eq:pwc-extension}
    \bar{u}_h(x) =
    \begin{cases}
      u_i,& x\in\KD_i,\ u_h\in\Gh,\\
      u_{ij},& x\in\LD_{ij},\ u_h\in\Fh.
    \end{cases}
  \end{equation}
\end{defn}

\begin{defn}[Piecewise linear extension]
  For a grid function $u_h\in\Gh$ we define its piecewise linear extension
  $\hat{u}_h$ as
  \begin{equation}\label{eq:pwl-extension-Gh}
    \hat{u}_h(x) =
      u_i + (\nabla_h u_h)_{i+1/2}(x-x_i),\quad x\in\KS_i,\ u_h\in\Gh,
  \end{equation}
  while for $u_h\in\Fh$ we define it as
  \begin{equation}\label{eq:pwl-extension-Fh}
    \hat{u}_h(x) =
      \begin{cases}
        u_{ij} + \delta_x^+u_{ij}(x-x_{i}) + (\nabla_{yh}u_h)_{i,j+1/2}(y-y_j),& x\in\LSDL_{ij},\\
        u_{i+1,j+1} + \delta_x^+u_{i,j+1}(x_{i+1}-x) +
        (\nabla_{yh}u_h)_{i+1,j+1/2}(y_j-y),& x\in\LSUR_{ij}.
      \end{cases}
  \end{equation}

\end{defn}

The following lemma shows the relation between discrete scalar products of grid
functions and scalar products of interpolated grid functions in $L^2(\Omega)$
and $L^2(\Omega\times Y)$ and follows by a direct calculation.

\begin{lem}\label{lem:extensions-props}
  It holds that
  \begin{align*}
    (\bar{u}_h,\bar{v}_h)_{L^2(\Omega)} &= (u_h,u_h)_{\Gh},& u_h,v_h&\in\Gh, \\
    (\nabla\hat{u}_h,\nabla\hat{v}_h )_{L^2(\Omega)} &= (\nabla_h u_h,\nabla_h v_h)_{\Eh},& u_h,v_h&\in\Gh, \\
    (\bar{u}_h,\bar{v}_h)_{L^2(\Omega\times Y)} &= (u_h,v_h)_{\Fh},& u_h,v_h&\in\Fh,\\
    (\nabla_y \hat{u}_h,\nabla_y\hat{v}_h)_{L^2(\Omega\times Y)} &= (\nabla_{yh} u_h,\nabla_{yh} v_h)_{\Hh},& u_h,v_h&\in\Fh.
  \end{align*}
\end{lem}

\subsection{Compactness}

In this subsection we prove our main result. To do this we essentially use the
preliminary results shown in the previous paragraphs and the results of
\cite{Ladyzhenskaya}. Basically, we show the convergence of semi-discrete
solutions to a weak solution of problem (P). This result is stated in the
following theorem.

\begin{thm} \label{thm:convergence}
  Assume (A1)--(A4) to be fulfilled. Then the semi-discrete solution
  $\{u_h^1$, $u_h^2$, $u_h^3$, $u_h^4\}$ of \eqref{eq:model} exists on $[0,T]$ for any $T>0$ and
  its interpolate $\{\hat{u}_h^1$, $\hat{u}_h^2$, $\hat{u}_h^3$,
  $\hat{u}_h^4\}$ converge in $L^2(\Omega)$, $L^2(\Omega\times Y)$,
  $L^2(\Omega\times S)$, $L^2(\Omega)$, respectively, as $|h|\to 0$ to a weak
  solution $(u_1$, $u_2$, $u_3$, $u_4)$ to problem (P) in the sense of Definition
  \ref{def:weak}.
\end{thm}
\begin{proof}
  We start off with recovering the initial data. The definition of interpolation
  of grid functions leads, as $|h|\to 0$, to
\begin{align*}
  \hat{u}^1_h(0) &\to u_1^0 \text{ weakly in } H^1(\Omega),\\
  \hat{u}^2_h(0) &\to u_2^0 \text{ weakly in } L^2(\Omega;H^1(Y)),\\
  \hat{u}^3_h(0) &\to u_3^0 \text{ weakly in } L^2(\Omega;H^1(Y)),\\
  \hat{u}^4_h(0) &\to u_4^0 \text{ weakly in } L^2(\Omega).
\end{align*}
%Hence, from the {\em a~priori} estimates it follows that
%\begin{align*}
%  u^1_h &\in L^\infty(0,T;\Gh),& \nabla_h u^1_h &\in L^2(0,T;\Eh),\\
%  u^2_h &\in L^\infty(0,T;\Fh),& \nabla_{yh} u^2_h &\in L^2(0,T;\Hh),\\
%  u^3_h &\in L^\infty(0,T;\Fh),& \nabla_{yh} u^3_h &\in L^2(0,T;\Hh),\\
%  u^4_h &\in L^\infty(0,T;\Gh),
%\end{align*}
%are bounded uniformly with respect to $h$. Therefore, $T_h$ is
%independent of $h$ and the solution of \eqref{eq:scheme} can be
%prolonged up to any $T>0$.

Let $h_n$ be a sequence of spatial space sizes such that $|h|\to 0$ as
$n\to\infty$. Consequently, we obtain a sequence of solutions
$\{u_{h_n}^1,u_{h_n}^2,u_{h_n}^3,u_{h_n}^4\}$ of \eqref{eq:scheme} defined on
the whole time interval $S$.

Let us pass to the limit $|h|\to 0$ in the ODE. Note that
$\eta(\bar{u}^3_{h_n}|_{y=\ell}, \bar{u}^4_{h_n} ) \weakto q$ weakly in
$L^2(S;L^2(\Omega))$, and $q$ still needs to be identified. The way we pass to
the limit in the ODE is based on the following monotonicity-type argument (see
\cite{Tasnim2scale}): using the monotonicity of $\eta$ w.r.t. both variables,
we can show that $\bar{u}_{h_n}^4$ is a Cauchy sequence, and therefore, it is
strongly convergent to $u^{4}$.

Now, it only remains to pass to the limit in the PDEs. Note that the weak
formulation contains a nonlinear boundary term involving $\eta(\cdot,\cdot)$.
Exploiting the properties of the interpolations of grid functions we deduce that
the same {\em a~priori} estimates hold also for the interpolated solution (see
also \cite{Ladyzhenskaya}). On this way, we obtain
\begin{align*}
  \{\hat{u}_{h_n}^1\} &\text{ is bounded in } L^\infty(0,T;L^2(\Omega)),\\
  \{\hat{u}_{h_n}^1\} &\text{ is bounded in } L^2(0,T;H^1(\Omega)),\\
  \{\hat{u}_{h_n}^2\} &\text{ is bounded in } L^\infty(0,T;L^2(\Omega)),\\
  \{\hat{u}_{h_n}^3\} &\text{ is bounded in } L^\infty(0,T;L^2(\Omega)),\\
  \{\hat{u}_{h_n}^4\} &\text{ is bounded in } L^\infty(0,T;L^2(\Omega)).
\end{align*}
Hence, there exists a subsequence of $h_n$ (denoted again by $h_n$), such that
\begin{align*}
  %\Sinterp u_{h_n}^1 &\weakto u^{1,1} \text{ weakly in } L^2(0,T;L^2(\Omega)),\\
  \hat{u}_{h_n}^1 &\weakto u^{1} \text{ weakly in } L^2(S;H^1(\Omega)),\\
  \hat{u}_{h_n}^2 &\weakto u^{2} \text{ weakly in } L^2(S;L^2(\Omega)),\\
  \hat{u}_{h_n}^3 &\weakto u^{3} \text{ weakly in } L^2(S;L^2(\Omega)),\\
  \hat{u}_{h_n}^4 &\weakto u^{4} \text{ weakly in } L^2(S;L^2(\Omega)).
\end{align*}

Since
$$
\|\hat{u}_{h_n}^1\|_{L^2(S,H^1(\Omega))} + \|\partial_t\hat{u}_{h_n}^1\|_{L^2(S,L^2(\Omega))}\leq C,
$$
Lions-Aubin's compactness theorem, see \cite[Theorem~1]{Lions}, implies
that there exists a subset (again denoted by $\hat{u}_{h_n}^1 $) such that
$$
\hat{u}_{h_n}^1 \longrightarrow u^1 \quad \text{strongly in} \quad L^2(S\times\Omega).
$$
To get the desired strong convergence for the cell solutions
$\hat{u}_{h_n}^2,\hat{u}_{h_n}^3$, we need the higher regularity with respect to
the variable $x$, proved in Lemma \ref{lem:improved-apriori}. We remark that the
two-scale regularity estimates imply that
$$
\|\hat{u}_{h_n}^2\|_{L^2(S;H^1(\Omega,H^1(Y)))} +
\|\hat{u}_{h_n}^3\|_{L^2(S;H^1(\Omega,H^1(Y)))} \leq C.
$$
Moreover, from Lemma \ref{lem:time-derivative-estimate}, we have that
$$
\|\partial_t\hat{u}_{h_n}^2\|_{L^2(S\times\Omega\times Y)} +
\|\partial_t\hat{u}_{h_n}^3\|_{L^2(S\times\Omega\times Y)} \leq C.
$$
Since the embedding
$$
H^1(\Omega, H^1(Y))\hookrightarrow L^2(\Omega, H^\beta(Y))
$$
is compact for all $\frac{1}{2} <\beta < 1$, it follows again from Lions-Aubin's
compactness theorem that there exist subsequences (again denoted
$\hat{u}_{h_n}^2,\hat{u}_{h_n}^3$), such that
\begin{equation}\label{strongconvHbeta}
  (\hat{u}_{h_n}^2,\hat{u}_{h_n}^3) \longrightarrow (u^2,u^3)\quad \text{strongly in}\quad L^2(S\times L^2(\Omega, H^\beta(Y)),
\end{equation}
for all $\frac{1}{2} <\beta < 1$. Now, (\ref{strongconvHbeta}) together with the
continuity of the trace operator
$$
H^\beta(Y) \hookrightarrow L^2(\partial Y), \quad \text{for} \quad \frac{1}{2}
<\beta < 1,
$$
yield the strong convergence of $\hat{u}_{h_n}^2$, $\hat{u}_{h_n}^3$ until the
boundary $y=0$.
\end{proof}

%%%%%%%%%%%%%%%%%%%%%%%%%%%%%%%%%%%%%%%%%%%%%%%%%%%%%%%%%%%%%%%%%%%%%%%%%%%%%%%%
\section{Numerical illustration of the two-scale FD  scheme}\label{illustration}

We close the paper with illustrating the behavior of the main chemical species
driving the whole corrosion process, namely of ${\rm H_2S(g)}$, and also the one
of the corrosion product -- the gypsum. To do these computations we use the
reference parameters reported in \cite{ChalupeckyEtAl2010}.

\begin{figure}[ht]
  \begin{center}
    \includegraphics[width=0.4885\textwidth]{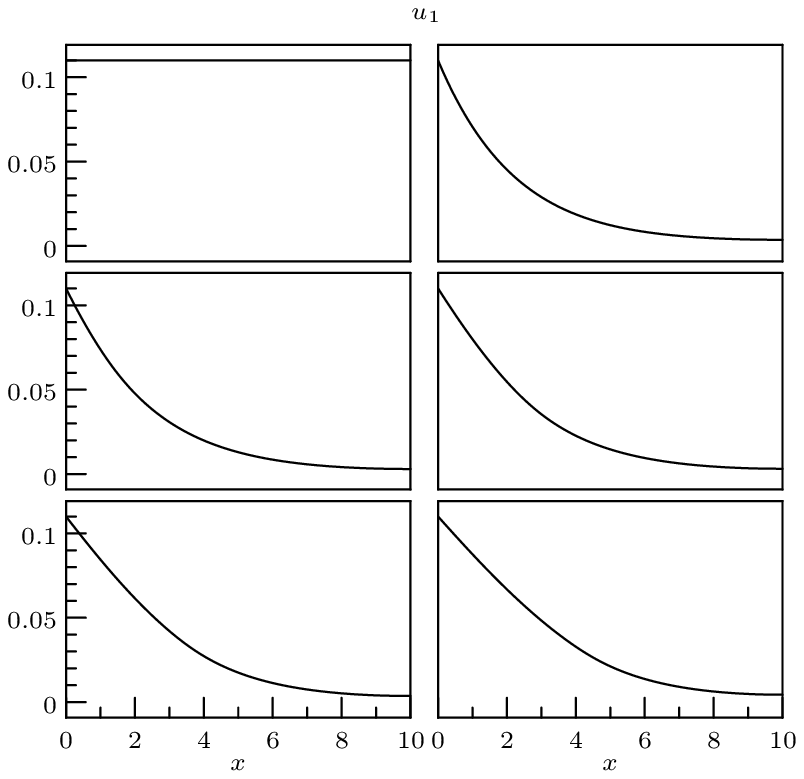}
    \includegraphics[width=0.48\textwidth]{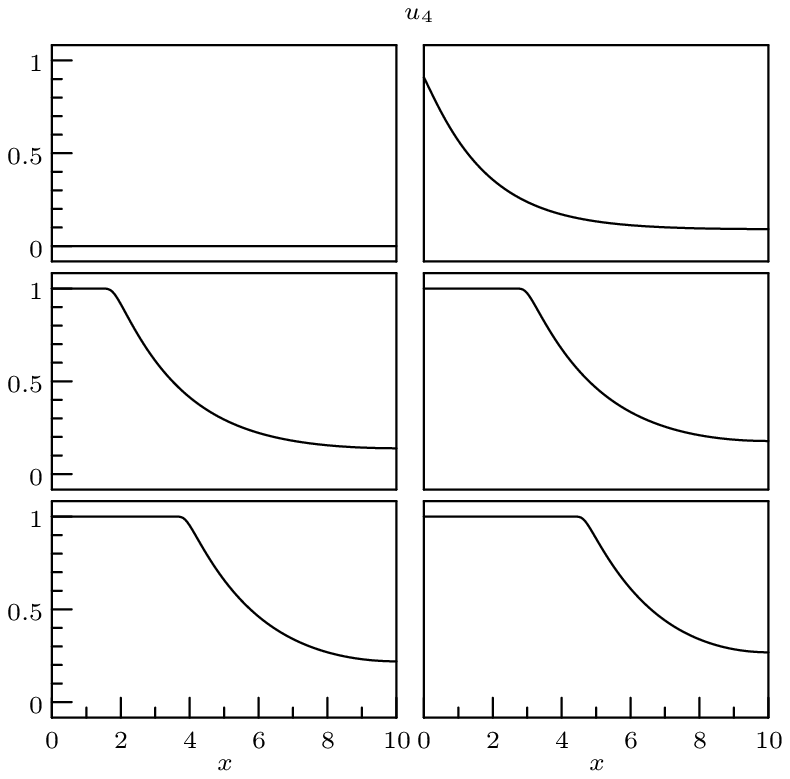}
  \end{center}
  \caption{Illustration of concentration profiles for the macroscopic
  concentration of gaseous $\mathrm{H_2S}$ (left) and of gypsum (right). Graphs
  plotted at times $t\in\{0, 80, 160, 240, 320, 400\}$ in a left-to-right and
  top-to-bottom order.}
  \label{fig:num-illust}
\end{figure}

Figure \ref{fig:num-illust} shows the evolution of $u_1(x,t)$ and $u_4(x,t)$ as
time elapses. Interestingly, although the behavior of $u_1$ is as expected
(i.e., purely diffusive), we notice that a {\em macroscopic} gypsum layer
(region where $u_4$ is produced) is formed (after a transient time $t^*>80$) and
grows in time. The figure clearly indicates that there are two distinct regions
separated by a slowly moving intermediate layer: the left region -- the place
where the gypsum production reached saturation (a threshold), and the right
region -- the place of the ongoing sulfatation reaction \eqref{eq:sulfatation}
(the gypsum production has not yet reached here the natural threshold). The
precise position of the separating layer is {\em a~priori unknown}. To capture
it simultaneously with the computation of the concentration profile would
require a moving-boundary formulation similar to the one reported in
\cite{BoehmEtAl1998}.

\paragraph*{Acknowledgements}
We acknowledge fruitful discussions with Maria Neuss-Radu (Heidelberg) and Omar
Lakkis (Sussex). V.~Ch. was supported by Global COE Program ``Education and
Research Hub for Ma\-the\-ma\-tics-for-Industry'' from the Ministry of Education,
Culture, Sports, Science and Technology, Japan. Part of this paper was written
while A.~M. visited the Faculty of Mathematics of the Kyushu University.

%%%%%%%%%%%%%%%%%%%%%%%%%%%%%%%%%%%%%%%%%%%%%%%%%%%%%%%%%%%%%%%%%%%%%%%%%%%%%%%%

\end{document}